\documentclass[11pt,reqno]{amsart} 
\usepackage{amssymb,epsfig,graphics}
\usepackage{enumerate}
\usepackage{color}
\def\r{\rightarrow}

\newcommand{\E}{\mathbb{E}}     
\renewcommand{\P}{\mathbb{P}}     
\renewcommand{\L}{\mathbb{L}}

\newcommand{\RR}{\mathbb{R}}
\newcommand{\EE}{\mathbb{E}}

\newcommand{\N}{\mathbb{N}}     

\newcommand{\R}{\mathbb{R}}     
     
\newcommand{\C}{\mathbb{C}}

\newcommand{\X}{\mathbb{X}}

\renewcommand{\ker}{\mathop{\rm Ker}}

\renewcommand{\r}{\mathop{\rightarrow}}

\newcommand{\cB}{\mathcal B}
\newcommand{\cC}{\mathcal C}
\newcommand{\cD}{\mathcal D}

\newcommand{\cL}{\mathcal L}

\newcommand{\cX}{\mathcal X}

\newtheorem{atheo}{Theorem}[section]
\newtheorem{adefi}[atheo]{Definition}
\newtheorem{alem}[atheo]{Lemma}
\newtheorem{arem}[atheo]{Remark}
\newtheorem{acor}[atheo]{Corollary}
\newtheorem{apro}[atheo]{Proposition}

\newtheorem{acond}[atheo]{Hypothesis}

\DeclareMathOperator{\Leb}{Leb}

\begin{document}

\title[Multiplicative ergodicity of Laplace transforms]{Multiplicative ergodicity of Laplace transforms for additive functional of Markov chains}
\author{Lo\"{\i}c Herv\'e}
\address{INSA de Rennes, F-35708, France; IRMAR CNRS-UMR 6625, F-35000, France; Universit\'e Europ\'eenne de Bretagne, France.}
\email{ Loic.Herve@insa-rennes.fr}
\author{Fran\c{c}oise P\`ene}
\address{Universit\'e de Brest and Institut Universitaire de France,
UMR CNRS 6205, Laboratoire de Math\'ematique de Bretagne Atlantique,
6 avenue Le Gorgeu, 29238 Brest cedex, France.}
\email{francoise.pene@univ-brest.fr}
\keywords{Markov processes, quasicompacity, operator, perturbation, ergodicity, Laplace transform}
\subjclass[2010]{Primary: 60J05}

\maketitle
\bibliographystyle{plain}
\vspace*{-8mm}
\begin{abstract}
We study properties of the Laplace transforms of non-negative additive functionals of Markov chains. We are namely interested in a multiplicative ergodicity property used in \cite{LouhichiYcart15} to study bifurcating processes with ancestral dependence. We develop a general approach based on the use of the operator perturbation method. We apply our general results to two examples of Markov chains, including 
a linear autoregressive model. In these two examples the operator-type assumptions reduce to some expected finite moment conditions on the functional (no exponential moment conditions are assumed in this work).
\end{abstract}
\date{\today}
\maketitle
\bibliographystyle{plain}
\tableofcontents




\section{Introduction}
In this work we study the Laplace transforms of non-negative additive
functionals of Markov chains by the use of the method of perturbation of operators. This method, introduced by Nagaev \cite{nag1,nag2} and by Le Page and Guivarc'h \cite{lep82,GuivarchHardy} to prove a wide class of limit theorems (central limit
theorem, local limit theorem, large and moderate deviations principles), has known an impressive development in the past decades (e.g.~see \cite{broi,hulo} and the references therein).
With the use of the classical operator perturbation method,
Laplace transforms of additive functionals of Markov chains
have been studied in many works. Let us mention namely \cite{KM03,KM05}.  These works, motivated by large deviations estimates,
require some exponential moment assumptions and the continuity of the family of
operators acting on the reference Banach space. \\
In the present work, we weaken these assumptions. Since we consider here non-negative observables, we do not require any exponential moment assumption.  
But the price to pay is that, in general, the classical perturbation method does not apply in our context to the family of Laplace operators (see Remark~\ref{rem-method} for details). Here we have to consider several Banach spaces instead of a single one. This is allowed by the Keller and Liverani perturbation theorem \cite{KelLiv99,Bal00}(e.g.~see \cite{HerPen10} and the references therein). The fact that we work with several spaces (due to our weak moment assumptions) complicates our study compared to the classical
approach.

Actually we study different properties of the Laplace transforms of non-negative additive functionals of Markov chains, namely their multiplicative ergodicity and the continuity and derivability of the radius of convergence of the Laplace-generating function, together with their spectral counterparts. We emphasize also on some applications of our result in the study of the bifurcating processes developed in [20]. 
The present work provides examples
coming from a markovian context satisfying some assumptions
of \cite{LouhichiYcart15}. We investigate in particular 
a multiplicative ergodicity property and its spectral analogous.

This paper is organized as follows. In Section \ref{results},
we introduce the notion of multiplicative ergodicity we are interested in, our notations and
we state our main results. We namely state general conditions
ensuring the multiplicative ergodicity of an additive functional of a
geometrically ergodic Markov chain.
We illustrate
our general result by two examples of Markov chains: the Knudsen gas model and some
linear autoregressive models. The proofs
for these examples are given in Sections
\ref{Knudsengas} and \ref{proofAR}. 
The more technical proofs of 
our general results are postponed in Appendix \ref{proofoperator} together with some other facts.
\section{Notations and main results}\label{results}
\subsection{Multiplicative ergodicity, examples}
Given a sequence $Y=(Y_n)_{n\ge 0}$ of non-negative valued random variables, we consider 
the generating function of the 
Laplace transforms of the partial sums of $Y$, that will be named
\emph{Laplace-generating function} of $Y$. 
We assume that the random variables $Y_n$ are not identically zero.
\begin{adefi}
\label{def:lapgen}
The Laplace transforms of the partial sums of
$Y$ are denoted by $L_Y^{(n)}$:
\begin{equation}
\label{eq:deflap}
\forall\gamma\in\mathbb R_+,\quad L_Y^{(n)}(\gamma) := \EE\left[\exp\left(-\gamma S_n\right)\right]\;,
\end{equation}
with $S_n:=\sum_{k=0}^nY_k$.
The Laplace-generating function of $Y$, denoted by $g_Y$, 
is the generating function of
the $L_Y^{(n)}(\gamma)$'s. For $\gamma,\lambda\in\RR^+$,
\begin{equation}
\label{eq:deflapgen}
g_Y(\gamma,\lambda) = 
\sum_{n=0}^{+\infty}\lambda^n L_Y^{(n)}(\gamma)\;.
\end{equation}
\end{adefi}
Observe that for all $\lambda\in[0,1)$, $g_Y(\cdot,\lambda)$ is 
non-increasing on $[0,+\infty)$, decreasing on 
$\{\gamma\ge 0\, :\, g_Y(\gamma,\lambda)<\infty\}$, 
starting at $g_Y(0,\lambda)=
1/(1-\lambda)$. Hence the radius of
convergence $R_Y(\gamma)$ of $g_Y(\gamma,\cdot)$ is 
non-decreasing in $\gamma$ from $R_Y(0)=1$. 
We are namely interested in the following properties:
\begin{equation}\label{P1}
\nu:=\inf\{\gamma>0\ :\ g_Y(\gamma,2)<\infty\}<\infty
\end{equation}
and
\begin{equation}\label{P2}
C_\nu:=\lim_{\gamma\rightarrow \nu+}\frac{\gamma-\nu}{\gamma}
      g_Y(\gamma,2)<\infty.
\end{equation}
In \cite{LouhichiYcart15}, it has been shown that these two properties imply the convergence in average of $e^{-\nu t}\E[N_t]$ where $N_t$ is the number of
cells at time $t$ in a mitosis process such that the life duration of 
the successive individuals of a same line has the distribution of $(Y_k)_k$.
To prove \eqref{P1} and \eqref{P2}, we will use the following notion of multiplicative ergodicity (see \cite{LouhichiYcart15}). Let 
us precise that the terminology "multiplicative ergodicity" is used
in the litterature with different levels of sharpness.
\begin{adefi} \label{def-mult-erg}
Let $\gamma_1>0$.
We say that $(S_n)_n$ is {\bf multiplicatively ergodic} on $J=[0,\gamma_1)$ if 
there exist two continuous maps $A$ and $\rho$ from $J$ to $(0,+\infty)$ such that,
for every compact subset $K$ of $(0,\gamma_1)$, there exist $M_K>0$ and $\theta_K\in(0,1)$ 
such that, for every $n\ge 1$, we have
$$ \sup_{\gamma\in K}|L_Y^{(n)}(\gamma)-A(\gamma) (\rho(\gamma))^n|\le M_K (\rho(\gamma)\theta_K)^n.$$
\end{adefi}
Observe that if $(S_n)_n$ is multiplicatively ergodic on $J=[0,\gamma_1)$,
then $\rho\equiv 1/R_Y$ and
$$\forall \gamma\in J,\ \forall \lambda>0,\quad 
 g_Y(\gamma,\lambda)<\infty\quad\Leftrightarrow\quad\lambda<\frac 1{\rho(\gamma)},$$
and, for every compact subset $K$ of $J$, we have
$$\forall \gamma\in K,\ \forall\lambda\in\left(0,\frac 1{\rho(\gamma)}\right),\quad
 \left\|g_Y(\gamma,\lambda)-\frac{A(\gamma)}{1-\lambda \rho(\gamma)}\right\|\le\frac{M_K}{1-\lambda \rho(\gamma)\theta_K}.$$
\begin{arem}\label{multergodP1P2}
If $(S_n)_n$  is multiplicatively ergodic,
then $\nu<\gamma_1$ means that
\begin{equation}\label{P1ter}
\nu=\inf\{\gamma\in J\ :\ \rho(\gamma)<1/2\}<\gamma_1.
\end{equation}
If moreover $\rho$ is differentiable at $\nu$ with $\rho'(\nu)\ne 0$, then
\eqref{P2} will follow with $C_\nu=-\frac {A(\nu)}{2\nu \rho'(\nu)}$. 
Actually, to obtain \eqref{P1ter}, we can relax the
continuity assumptions on $A$ and $\rho$ on $J=[0,\gamma_1)$. For \eqref{P2}, we just need the continuity of $A$ and the differentiability of $\rho$ at $\nu$.
\end{arem}
We focus our study on the three following properties: 
\vspace*{-2mm}
\begin{itemize}
\item the geometric ergodicity on some maximal interval $[0,\gamma_1)$,
\item\eqref{P1} and more generally the study of $\lim_{\gamma\rightarrow\gamma_1}\rho(\gamma)$,
\item \eqref{P2} and more generally the differentiability of $\rho$ on $(0,\gamma_1)$ and
the fact that $\rho'<0$ on this interval.
\end{itemize}
We investigate these properties in the context of additional functional of Markov chains. Let $(\X,\cX)$ be a measurable space, let $(X_n)_n$ be a Markov chain on $(\X,\cX)$ 
with Markov kernel $P(x,dy)$ and invariant probability $\pi$, and let $\xi : \X\r [0,+\infty)$ be a measurable function. Recall that $\xi$ is said to be coercive if $\lim_{|x|\rightarrow +\infty}\xi(x)=+\infty$, i.e.~if, for every $\beta$,  $[\xi\le\beta]$ is bounded. 
Moreover we consider
$$S_n:=\sum_{k=0}^n\xi(X_k).$$ 
We identify $X$ with the canonical Markov chain and write $\P_\mu$
for the probability measure corresponding to the case when
the initial probability distribution (i.e. the distribution of $X_0$) is $\mu$.
For every $x\in\X$, we simply write $\P_x$ when $\mu=\delta_x$. We write $\E_\mu[\cdot]$ and $\mathbb E_x[\cdot]$ for the corresponding expectations.
We then write 
$$\rho_{Y,\mu}(\gamma):=\limsup_{n\rightarrow +\infty}(\mathbb E_\mu[e^{-\lambda S_n}])^{\frac 1n} \quad \text{and } \quad \rho_{Y,x}(\gamma):=\limsup_{n\rightarrow +\infty}(\mathbb E_x[e^{-\lambda S_n}])^{\frac 1n}.$$
Moreover we simply write $\rho_Y$ for $\rho_{Y,\mu}$ in the case when $\rho_{Y,\mu}$ does not depend on the initial distribution. 
In this context we develop a general method to prove the
multiple ergodicity and even a spectral version of this property.
As a consequence, we prove the following result.
\begin{atheo}[Linear autoregressive model]\label{thmAR}
Assume that $\X:=\R$ and $(X_n)_{n\in\N}$ is the linear autoregressive model defined by $X_n = \alpha X_{n-1} + \vartheta_n$ for $n\ge 1$,
where $X_0$ is a real-valued random variable, $\alpha\in(-1,1)$, and $(\vartheta_n)_{n\ge 1}$ is a sequence of  i.i.d. real-valued random variables, admitting a moment of order $r_0>0$, independent of $X_0$.  Assume that $\vartheta_1$ has a continuous Lebesgue probability density function $p>0$ on $\X$ such that
$$\forall x\in \mathbb R,\ \ \exists\varepsilon_x>0,\ \ \int_{\mathbb R} \sup_{|z|<\varepsilon_x}
     p(y+x+z)\, dy<\infty.$$
Assume moreover that $\xi$ is continuous and coercive, that $\xi(x)>0$ for Lebesgue almost every $x\in\mathbb R$ and that $\sup_{x\in\mathbb R}\frac{\xi(x)}{(1+|x|)^{r_0}}<\infty$.

Then, $\rho_{Y,x}(\gamma)=\rho_{Y,\pi}(\gamma)$, $(S_n)_n$ is multiplicatively ergodic on $(0,+\infty)$ 
with respect
to $\P_\pi$ and to $\P_x$ for any $x\in\X$. 
Furthermore $\lim_{\gamma \rightarrow+\infty} \rho_{Y,\pi}(\gamma)=0$. Hence \eqref{P1} holds true under $\mathbb P_\pi$
or $\mathbb P_x$ for any $x\in\X$.

If moreover there exists $\tau>0$ such that  $\sup_{x\in\mathbb R}\frac{\xi(x)^{1+\tau}}
{(1+|x|)^{r_0}}<\infty$, then 
$\rho_{Y,\pi}$ is differentiable and admits a negative derivative on $(0,+\infty)$ and so \eqref{P2} holds also true  under $\mathbb P_\pi$
or $\mathbb P_x$  for any $x\in\X$.
\end{atheo}
We also prove the following result for the simple example of Knudsen gas.
\begin{atheo}[Knudsen gas] \label{pro-Knudsen}
Let $\X:=\R^d$, $\pi$ be some Borel probability measure on $\X$. Given $\alpha\in(0,1)$ and a Markov kernel $U$ on $\R^d$ with stationary measure $\pi$, 
we consider the canonical 
Markov chain $X$ with transition kernel $P$ given by 
$P = \alpha \pi + (1-\alpha)\, U$.

Then $(S_n)_n$ is multiplicatively ergodic on the interval $J_0=\{\gamma>0\, :\, r(\gamma)>1-\alpha\}$ with respect to $\mathbb P_\mu$ for any probability distribution $\mu$ on $\X$
absolutely continuous with respect to $\pi$, with density in $\L^{p}(\pi)$
for some $p>1$. 

Assume $\alpha>1/2$ and $2\alpha \sum_{n\ge 0}(2(1-\alpha))^n \P_\pi\left(\sum_{k=0}^{n}Z_k=0\right)<1$, where $(Z_n)_n$ is a Markov process with transition $U$, then \eqref{P1}
holds with respect 
to $\P_\pi$ and 
to $\P_\mu$ for every probability distribution $\mu$ in $\X$
satisfying the previous conditions (this is true in particular if $\alpha>1/2$ and $\pi(\xi=0)=0$).

Assume moreover that $\pi(\xi^\tau)<\infty$ for some $\tau>1$. Then
\eqref{P2} holds also true  with respect to $\P_\pi$ and to 
$\P_\mu$ for every probability distribution $\mu$
admitting a density with respect to $\pi$ which belongs to 
$\L^{p}(\pi)$ for some $p>\frac{\tau}{\tau-1}$.
\end{atheo}
\begin{arem} \label{rem-method}
Now let us say a few words about our general approach. 
We will consider the family of perturbed operators $(P_\gamma:=P(e^{-\gamma\xi}\cdot))_{\gamma>0}$ acting on some Banach spaces of measurable functions (or of classes of measurable functions). 
For linear autoregressive models (Theorem \ref{thmAR}), we will work with Banach spaces
$\mathcal B_a=\mathcal C_{V^a}$ linked to the weighted-supremum Banach spaces.
For the Knudsen gas (Theorem \ref{pro-Knudsen}), we will work with $\mathcal B_a=\mathbb L^{a}(\pi)$. 
Because we do not assume any exponential moment condition on $\xi$ (contrarily to the papers mentioned in Introduction), the map $\gamma\mapsto P_\gamma$ is not continuous from $(0,+\infty)$ to $\mathcal L(\mathcal B_{a})$, but only
from $(0,+\infty)$  to $\mathcal L(\mathcal B_a,\mathcal B_b)$ for $a<b$ for the linear autoregressive models (and for $b<a$ for the Knudsen gas).
For this reason, the classical operator perturbation method  \cite{nag1,GuivarchHardy}
(see also \cite{hulo} and the references therein) does not apply to our context.
But its improvement given by the Keller-Liverani
perturbation theorem \cite{KelLiv99} will be appropriate to our purposes.
\end{arem}

\subsection{Notations}\label{nota}
For any normed complex vector spaces $(\cB_0,\|\cdot\|_{\cB_0})$ and $(\cB_1,\|\cdot\|_{\cB_1})$, the set of continuous $\mathbb C$-linear operators from $\cB_0$ to $\cB_1$
will be written $\mathcal L(\cB_0,\cB_1)$. This set is endowed with the operator norm $\|\cdot\|_{\cB_0,\cB_1}$ given by
$$\forall Q\in\mathcal L(\cB_0,\cB_1),\ \  \|Q\|_{{\cB_0},{\cB_1}}=\sup_{f\in\cB_0,\ \|f\|_{{\cB_0}}=1}\|Qf\|_{{\cB_1}}.$$
The notation $\mathcal B_0\hookrightarrow \mathcal B_1$ means that $\cB_0$ is continuously injected in $\mathcal B_1$. 

If $\cB$ is a complex Banach space, we will simply write $(\cB^*,\|\cdot\|_{\cB^*})$ 
for the topological dual space 
$(\mathcal L(\cB,\mathbb C),\|\cdot\|_{\cB,\mathbb C})$
of $\cB$ and $(\mathcal L(\cB),\|\cdot\|_{\cB})$ for $(\mathcal L(\cB,\cB),\|\cdot\|_{\cB,\cB})$.
For any $Q\in\cL(\cB)$, we denote by $Q^*$ its adjoint operator. 
We write $\sigma(Q)=\sigma(Q_{|\cB})$ for the spectrum of $Q$:
$$\sigma(Q) :=\{\lambda\in\mathbb C\ :\ (Q-\lambda\, I)\mbox{ is non invertible}\},$$
where $I$ denotes the identity operator on $\cB$. Recall that $Q$ and $Q^*$ have the same norm in $\cL(\cB)$ and $\cL(\cB^*)$ respectively, as well as the same spectrum. We write
$r(Q)=r(Q_{|\cB})$ for the spectral radius of $Q$:
$$r(Q_{|\cB}):=\sup\{|\lambda|,\ \lambda\in\sigma(Q)\}=\lim_n\|Q^n\|_{\cB}^{1/n} $$
and $r_{ess}(Q)=r_{ess}(Q_{|\cB})$ for its essential spectral radius:
$$r_{ess}(Q):=\lim_n \inf_{F\in\cL(\cB)\ \mbox{\scriptsize compact}}\|Q^n-F\|^{1/n}_{\cB}.$$
Recall that we also have
$$r_{ess}(Q):=\sup\{|\lambda|\ :\ \lambda\in\mathbb C\ \mbox{and} \ (Q-\lambda\,  I)\mbox{ is non Fredholm}\}.$$

Let $(\X,\cX)$ be a measurable space, let $X=(X_n)_n$ be a Markov chain on $(\X,\cX)$ 
with Markov kernel $P(x,dy)$ and invariant probability $\pi$, and let $\xi : \X\r [0,+\infty)$ be a measurable function. 
We then consider
$$Y_k:=\xi(X_k)\quad\mbox{and}\quad S_n:=\sum_{k=0}^nY_k.$$
We identify $X$ with the canonical Markov chain.
We consider the nonnegative kernels $P_\gamma(x,dy)$ defined by
\begin{equation}
\forall \gamma\in[0,+\infty),\quad P_\gamma(x,dy) := e^{-\gamma\xi(y)}\, P(x,dy)\ \ \mbox{and}\ \ 
       P_{\infty}(x,dy):=\mathbf 1_{\{\xi=0\}}(y)\, P(x,dy).
\end{equation}
We use the same notations $P_\gamma$ for the linear operators associated
with these kernels:
$$\forall x\in\X,\quad (P_\gamma f)(x) :=\int_{\X}f(y)\, P_\gamma(x,dy) .$$

In the sequel $P_\gamma$ will be assumed to continuously act on a (or several) Banach space $\cB$. Such a space will contain $\mathbf 1_\X$ and $\pi$ will be in its topological dual space. Moreover we write $r(\gamma) := r(P_{\gamma|\cB})$ for the spectral radius of $P_\gamma$. With these notations, we have
\begin{equation}\label{formuleFourier}
\mathbb E_{\mu}[e^{-\gamma S_n}]=\mu(e^{-\gamma\xi}
P_\gamma^n \mathbf 1_{\X})\quad\mbox{and}\quad g_Y(\gamma,\lambda)=\mu\big(e^{-\gamma\xi}(I-\lambda P_\gamma)^{-1}\mathbf{1}_\X\big),
\end{equation}
for any $\lambda<\frac 1{r(\gamma)}$ and 
for any initial distribution $\mu$ on $\X$ such that $f\mapsto\mu(e^{-\gamma \xi}f)$ belongs to $\cB^*$. 

Let us now recall the definition of Banach lattice spaces of functions (or classes of functions modulo $\pi$), which will be used in our examples to obtain the expected nonincreasingness of $r(\cdot)$ and some suitable spectral properties for $P_\gamma$.
\begin{adefi}
A complex Banach space $(\cB,\|\cdot\|_{\cB})$ of functions $f:\X\rightarrow\mathbb C$
(or of classes of such functions modulo $\pi$) is said to be a {\bf complex Banach lattice} if it is stable by $|\cdot|$, by real part and if
$$\forall f,g\in\cB,\quad f(\X)\cup g(\X)\subset\R\quad
      \Rightarrow\quad \min(f,g),\, \max(f,g)\in\cB , $$
$$\forall f,g\in\cB,\quad |f|\le|g|\quad\Rightarrow\quad \|\, |f|\, \|_\cB = \|f\|_\cB\le \|g\|_\cB = \|\, |g|\, \|_\cB.$$
\end{adefi} 
 Classical instances of Banach lattices of functions are the spaces $(\L^p(\pi),\|\cdot\|_p)$ and $(\cB_V,\|\cdot\|_{V})$ (see (\ref{def-La}) and (\ref{def-BV})), as well as the space $(\mathcal L^\infty(\mathbb X),\|\cdot\|_{\infty})$ composed of all the bounded measurable $\C$-valued functions on $\X$, and equipped with its usual norm $\|f\|_{\infty}:=\sup_{x\in\X}|f(x)|.$
\subsection{General results}
We first prove that the monotonicity of 
$\gamma\mapsto r(\gamma) := r(P_{\gamma|\cB})$ is easy to establish when $\cB$ is a Banach lattice of functions.
\begin{alem}\label{LEMME0}
If $(\cB,\|\cdot\|_{\cB})$ is a complex Banach lattice of functions $f:\X\rightarrow\mathbb C$ (or of classes of functions modulo $\pi$), then the map $\gamma\mapsto r(\gamma)$ is non increasing on $[0,+\infty)$. 
\end{alem}
\begin{proof}
For any $0\le\gamma<\gamma'\le\infty$ and for any $f,g\in\cB$ such that 
$|f|\le|g|$, we have $e^{-\gamma'\xi} |f|\le e^{-\gamma\xi} |g|$ and so
$P_{\gamma'}|f|\le P_{\gamma}|g|$, which implies by induction that $P_{\gamma'}^n|f|\le P_{\gamma}^n|f|$ for every integer $n\ge 1$. We conclude that $\|P_{\gamma'}^n\|_{\cB}\le \|P_{\gamma}^n\|_{\cB}$
since $(\cB,\|\cdot\|_{\cB})$ is a Banach lattice.
This implies that $r(\gamma')\le r(\gamma)$ and so the desired statement.
\end{proof}
However it can happen that 
$r(\gamma)=1$ for every $\gamma\in[0,+\infty)$ (see Appendix \ref{counterexample}). 
To study the multiplicative ergodicity as well as regularity properties of $\gamma\mapsto r(\gamma) := r(P_{\gamma|\cB})$, where $\cB$ is a Banach space on which $P_\gamma$ continuously acts, we use
the Keller-Liverani perturbation theorem \cite{KelLiv99}. 
This result brings a significative improvement to the classical
Nagaev-Guivarc'h perturbation method \cite{nag1,nag2,GuivarchHardy,GuivarchLepage}.
Indeed, it enables the study of spectral properties of family of operators
$(Q(t))_{t\in J}$ acting on $\mathcal B_i$ such that 
$t\mapsto Q(t)$ fails to be continuous from $J$ to $\mathcal L(\mathcal B_i)$ but is continuous from $J$ to $\mathcal L(\mathcal B_i,\mathcal B_{i+1})$.
\begin{acond} \label{hypKL}
Let $\cB_0$ and $\cB_1$ be two Banach spaces, let $J$ be 
a subinterval of $[-\infty,+\infty]$, and let $(Q(t))_t$ be a family of operators. We will say that $((Q(t))_t,J,\cB_0,\cB_1)$ satisfies Hypothesis~\ref{hypKL} if
\begin{itemize}
\item $\cB_0 \hookrightarrow \cB_1$,
\item for every $t\in J$, $Q(t)\in\cL(\cB_0)\cap\cL(\cB_1)$,
\item the map $t\mapsto Q(t)$ is continuous from $J$ to
$\cL(\cB_0,\cB_1)$,
\item there exist $c_0>0$, $\delta_0>0$, $M>0$ such that 
\begin{subequations}
\begin{equation} \label{cond-r-ess-direct}
\forall t\in J,\quad r_{ess}\big(Q(t)_{|\cB_0}\big)\le\delta_0
\end{equation}
\begin{equation} \label{D-F-direct}
\forall t\in J,\ \forall n\geq 1,\ \forall f\in\cB_0,\quad 
\|Q(t)^n f\|_{\cB_0}\le c_0\big(\delta_0^n\| f\|_{\cB_0}+M^n\| f\|_{\cB_1}\big)
\end{equation}
\end{subequations}
\end{itemize}
\end{acond}
\noindent{\bf Hypothesis~\ref{hypKL}*.} 
{\it $((Q(t))_t,J,\cB_0,\cB_1)$ satisfies all the conditions of Hypothesis~\ref{hypKL}, except for (\ref{cond-r-ess-direct}) and (\ref{D-F-direct}) which are replaced by the following ones: 
\begin{subequations}
\begin{equation} \label{cond-r-ess-dual}
\forall t\in J,\quad r_{ess}\big(Q(t)^*_{\ |\cB_1^*}\big)\le\delta_0
\end{equation}
\begin{equation} \label{D-F-dual}
\forall t\in J,\ \forall n\geq 1,\ \forall f^*\in\cB_1^*,\quad 
\|(Q(t)^*)^n f^*\|_{\cB_1^*}\le c_0(\delta_0^n\| f^*\|_{\cB_1^*} + M^n\| f^*\|_{\cB_0^*})
\end{equation}
\end{subequations}
}
\begin{arem}\label{KL-remarque}
Hypothesis~\ref{hypKL} contains the conditions of the Keller-Liverani perturbation theorem \cite{KelLiv99} when applied to the family $\{Q(t),\, t\in J\}$ with respect to the spaces $\cB_0 \hookrightarrow \cB_1$. Hypothesis~\ref{hypKL}* contains the conditions of the Keller-Liverani theorem when applied to the family $\{Q(t)^*,\, t\in J\}$ with respect to $\cB_1^* \hookrightarrow \cB_0^*$. Indeed observe that the three first conditions of Hypothesis~\ref{hypKL}, which are assumed in Hypothesis~\ref{hypKL}*,  are equivalent to the following ones: $\cB_1^* \hookrightarrow \cB_0^*$, for every $t\in J$ we have $Q(t)^*\in\cL(\cB_0^*)\cap\cL(\cB_1^*)$, and finally $t\mapsto Q(t)^*$ is continuous from $J$ to $\cL(\cB_1^*,\cB_0^*)$. But it is worth noticing that the conditions (\ref{cond-r-ess-dual})-(\ref{D-F-dual}) cannot be deduced from (\ref{cond-r-ess-direct})-(\ref{D-F-direct}) (and conversely). 
\end{arem}

Let us now state the Keller-Liverani perturbation theorem in our context. 

\begin{atheo}[\cite{KelLiv99}]\label{thmkellerliverani1}
Under Hypothesis \ref{hypKL} (respectively under Hypothesis~\ref{hypKL}*) the function $t\mapsto r(t):=r((Q(t))_{|\cB_0})$ (respectively $t\mapsto r(t):=r((Q(t))_{|\cB_1})$) is continuous on the set $\{t\in J : r\big(Q(t)_{|\cB_0}\big)>\delta_0\}$ (respectively on $\{t\in J : r\big(Q(t)^*_{\ |\cB_1^*}\big)>\delta_0\}$). Moreover, in both cases, the following inequality holds: 
$$\forall t_0\in J,\quad \limsup_{t\rightarrow t_0}r(t)\le \max(\delta_0,r(t_0)).$$ 
\end{atheo}

Given a Banach space $\cB$ of functions on $\X$ (or of classes of such functions modulo $\pi$), 
recall that  $\psi\in\mathcal B ^*$ is said to be non-negative if $\psi(f)\ge 0$ for every $f\in\mathcal B$, $f\ge 0$. 
Let us introduce another assumption. 
\begin{acond}\label{hypcompl}
Let $\gamma\in[0,+\infty]$ and $\mathcal B$ be a Banach lattice 
of functions on $\X$ (or of classes of such functions modulo $\pi$). Assume that $1_\X\in\cB\subset \mathbb L^1(\pi)$, 
that $P_\gamma$ is quasi-compact on $\cB$ with spectral radius $r(\gamma):=r(P_{\gamma|\cB})>0$, and that
\begin{itemize}
\item if $\phi\in\cB$ is non-null and non-negative, then $P_\gamma \phi > 0$ (modulo $\pi$) and, for every non-null non-negative $\psi\in\cB^*\cap\ker(P_\gamma^*-r(\gamma)I)$, we have $\psi(P_\gamma\phi)>0$.
\item for every $f,g\in\cB$ with $f>0$, $P_\gamma f=r(\gamma)f$ and $P_\gamma g=r(\gamma)g$, we have $g\in \C\cdot f$,
\item 1 is the only complex number $\lambda$ of modulus 1 such that
$P(h/|h|)=\lambda h/|h|$  in $\L^1(\pi)$ for some $h\in\cB$, $|h|>0$, modulo $\pi$.
\end{itemize}
\end{acond}
Now we state general conditions ensuring namely the 
multiplicative ergodicity and the needed regularity properties of $r$. Under Hypothesis \ref{hypKL} or \ref{hypKL}* we define the following set: 
$$J_0:=\{t\in J : r(\gamma)>\delta_0\}.$$
\begin{atheo} \label{generalspectraltheorem1}
Let $\cB_0\hookrightarrow\cB_1\hookrightarrow \L^1(\pi)$ be two Banach spaces and let $J$ be a subinterval of $[0,+\infty]$. Assume that $(P_\gamma,J,\cB_0,\cB_1)$ satisfies Hypothesis \ref{hypKL} or \ref{hypKL}* and that 
\begin{itemize}
	\item Hypothesis \ref{hypcompl} holds with $J$ and $\cB :=\cB_0$ under Hypothesis \ref{hypKL}
	\item Hypothesis \ref{hypcompl} holds with $J$ and $\cB :=\cB_1$ under Hypothesis \ref{hypKL}*. 
\end{itemize}
Then $\gamma\mapsto r(\gamma) := r(P_{\gamma|\cB})$ is continuous on $J_0$, and there exists a map $\gamma\mapsto \Pi_\gamma$ from $J_0$ to $\cL(\cB)$ which is continuous from $J_0$ to 
$\mathcal L(\cB_0,\cB_1)$ such that, for every compact subset $K$ of $J_0$, there exist $\theta_K\in(0,1)$ and $M_K\in(0,+\infty)$ such that 
\begin{equation} \label{sup-vit}
\forall \gamma\in K,\ \forall f\in\cB,\quad 
\big\|(P_\gamma^n(f)) - r(\gamma)^n \Pi_\gamma f\big\|_{\cB} \leq M_K\, \big(\theta_K\, r(\gamma)\big)^n \|f\|_{\cB}.
\end{equation}
\end{atheo}
Under the assumptions of  Theorem \ref{generalspectraltheorem1}, we obtain for every $f\in\cB_0$ and for every $\psi\in\cB_1^*$:  
\begin{equation} \label{sup-vit-bis}
\forall \gamma\in K,\ \quad 
\big|\psi(P_\gamma^n(f)) - r(\gamma)^n \psi\big(\Pi_\gamma f\big)\big| \leq M_K\, \big(\theta_K\, r(\gamma)\big)^n \|\psi\|_{\cB_1^*}\|f\|_{\cB_0}.
\end{equation}
with $\gamma\mapsto r(\gamma)$ and $\gamma\mapsto \psi(\Pi_\gamma f)$ continuous from $J_0$ to $[0,1]$ and $\C$ respectively. 
Then \eqref{sup-vit} and \eqref{sup-vit-bis} can be interpreted as
{\bf spectral multiplicative ergodicity properties}.

Theorem~\ref{generalspectraltheorem1} is established in Appendix~\ref{ap-proof-general-th}. Since $1_\X\in\cB$ (see Hypothesis \ref{hypcompl}) with $\cB :=\cB_0$ or $\cB :=\cB_1$ according that  Hypothesis~\ref{hypKL} or \ref{hypKL}* is assumed, the following corollary easily follows from Theorem~\ref{generalspectraltheorem1}.   
\begin{acor} \label{cor-th1}
Assume that the assumptions of  Theorem \ref{generalspectraltheorem1} hold and
that $f\mapsto e^{-\gamma\xi}f$ is in $\mathcal L(\cB_1)$.
Let $\mu$ be a probability measure on $\mathbb X$ belonging to $\cB_1^*$ and satisfying: $\forall\gamma\in J_0,\ \mu(e^{-\gamma\xi}\Pi_\gamma\mathbf 1_{\X})>0$.
Then, under $\mathbb P_\mu$, $\rho_Y(\gamma)=r(\gamma)$ and the
sequence $(S_n)_n$ 
is multiplicatively ergodic on $J_0$ (with $A(\gamma)=\mu(e^{-\gamma\xi}\Pi_\gamma\mathbf 1_{\X})$). If moreover $\inf r(J_0) <1/2<\sup r(J_0)$, then $\nu$ is finite and is given by 
\begin{equation}\label{P1bis}
\nu=\inf\{\gamma>0\, :\, r(\gamma)<1/2\}.
\end{equation}
\end{acor}
Now we are interesting in the differentiability of $r$ and in the sign of its derivative under assumptions analogous to those of Theorem \ref{generalspectraltheorem1}. 
\begin{atheo}\label{generalspectraltheorem2}
Assume $\pi(\xi>0)>0$. Let $\cB_0\hookrightarrow\cB_1\hookrightarrow\cB_2\hookrightarrow \cB_3 \hookrightarrow \L^1(\pi)$ be Banach spaces and let $J$ be a subinterval of $[0,+\infty]$. Assume that one of the two following conditions holds 
\vspace*{-2mm}  
\begin{enumerate}[(a)]
\item Either: for $i=0,1,2$, $(P_\gamma,J,\cB_i,\cB_{i+1})$ satisfies Hypothesis \ref{hypKL}, and Hypothesis~\ref{hypcompl} holds with $(J,\cB_{i})$  ; in this case we set $\cB:=\cB_0$. 
\item Or: for $i=0,1,2$, $(P_\gamma,J,\cB_i,\cB_{i+1})$ satisfies Hypothesis \ref{hypKL}*, and Hypothesis \ref{hypcompl} holds with $(J,\cB_{i+1})$ ; in this case we set $\cB:=\cB_3$. 
\end{enumerate}
Moreover assume that $\gamma\mapsto P_\gamma$ is continuous from $J$ to $\cL(\cB_i,\cB_{i+1})$ for $i\in\{0,2\}$ and $C^1$ from $J$ to $\cL(\cB_1,\cB_2)$ with derivative $P'_\gamma f=P_\gamma(-\xi f)$ ($f\mapsto \xi f$ being in $\cL(\cB_1,\cB_2)$). 

Then $\gamma\mapsto r(\gamma) := r(P_{\gamma|\cB})$ is $C^1$
on $J_0:=\{t\in J : r(\gamma)>\delta_0\}$ with negative derivative, $\gamma\mapsto\Pi_\gamma$ is well defined from $J_0$ to $\cL(\cB)$ and is $C^1$ from $J_0$ to $\cL(\cB_0,\cB_3)$. 
\end{atheo}
Theorem~\ref{generalspectraltheorem2} is proved in Appendix~\ref{ap-proof-general-th}. The next corollary follows from Theorem~\ref{generalspectraltheorem2}. 
\begin{acor} \label{cor-th2}
Assume that the assumptions of Theorem  \ref{generalspectraltheorem2} hold and that
$f\mapsto e^{-\gamma\xi}f$ is in $\mathcal L(\cB_3)$.
Let $\mu$ be a probability measure on $\mathbb X$ belonging to $\cB_3^*$ and satisfying: $\forall\gamma\in J_0,\ \mu(e^{-\gamma\xi}\Pi_\gamma\mathbf 1_{\X})>0$.
 then $(S_n)_n$
is multiplicatively ergodic on $J_0$ with respect to $\mathbb P_\mu$. If moreover $\inf r(J_0) <1/2<\sup r(J_0)$, then $\nu$ is finite and is given by \eqref{P1bis} and $C_\nu$ of \eqref{P2} is well defined and finite.
\end{acor}

\begin{arem} $\ $ \\
\vspace*{-8mm}
\begin{itemize}
  \item In the previous statements, the Banach lattice assumption in   Hypothesis~\ref{hypcompl} can be replaced by: $r(\gamma)$ is a pole of finite order of $P_\gamma$.
	\item As already mentioned, the (expected) nonincreasingness of $r(\cdot)$ is guaranteed since our spaces are assumed to be Banach lattices (see Lemma \ref{LEMME0}). Note that this implies that $J_0$ is an interval. 
 \item We will see in Appendix~\ref{proofoperator} that $P_\gamma(\Pi_\gamma\mathbf 1_\X) = r(\gamma) \Pi_\gamma\mathbf 1_\X$. Then we deduce from the first condition in  Hypothesis~\ref{hypcompl} that $\Pi_\gamma\mathbf 1_\X > 0\ \, \pi-$a.s., where $\pi$ is the stationary distribution of $(X_n)_n$. Consequently, if the assumptions of Corollary~\ref{cor-th1} (respectively Corollary~\ref{cor-th2}) are fulfilled, then its conclusions hold true with $\mu=\pi$ since  $\pi\in\cB_1^*$  since  $\cB_1\hookrightarrow \L^1(\pi)$ 
(resp. $\pi\in\cB_3^*$  since  $\cB_3\hookrightarrow \L^1(\pi)$)
and $\pi(e^{-\gamma\xi}\Pi_\gamma\mathbf 1_{\X})>0$. This is also true for any probability measure $\mu\in\cB_1^*$ (respectively $\mu\in\cB_3^*$) which is absolutely continuous with respect to $\pi$. If $\Pi_\gamma\mathbf 1_\X > 0$ everywhere, then the conclusions  of Corollary~\ref{cor-th1} (respectively Corollary~\ref{cor-th2}) hold for any $\mu\in\cB_1^*$  (respectively $\mu\in\cB_3^*$). 
	\item In Case $(a)$ of Theorem  \ref{generalspectraltheorem2} we will prove in Appendix that, for every $\gamma\in J_0$, the spectral radius $r(P_{\gamma| \cB_i})$ does not depend on $i\in\{0,1,2\}$, and that (\ref{sup-vit}) holds on $\cB_i$ for every $i=0,1,2$. In Case $(b)$ the same properties hold for $i=1,2,3$. 
\end{itemize}
\end{arem}

We conclude this section with some complementary results which may be useful. 
\begin{acor} \label{cor-th1v1}
Assume that the assumptions of Theorem  \ref{generalspectraltheorem1} hold 
and that
$f\mapsto e^{-\gamma\xi}f$ is in $\mathcal L(\cB_1)$.
Let $\mu$ be a probability measure on $\mathbb X$ belonging to $\cB_1^*$ and satisfying: $\forall\gamma\in J_0,\ \mu(e^{-\gamma\xi}\Pi_\gamma\mathbf 1_{\X})>0$.  
\begin{enumerate}[(i)]
	\item If $J=[0,+\infty]$ and if $\alpha_0$ is a positive real number such that $\alpha_0 <1/\delta_0$, then under $\mathbb P_\mu$,
for every $\gamma\in [0,+\infty]$,
we have
$$ g_{Y,\mu}(\gamma,\alpha_0)<\infty\ \Leftrightarrow\ r (\gamma)<1/{\alpha_0}.$$
  \item If  $J=[0,+\infty]$ and $\delta_0 <1/2$,
then $\nu<\infty\ \Leftrightarrow\ r(\infty)<1/2$.
\end{enumerate}
\end{acor}
\begin{proof}
Observe first that, for every $\gamma\in J_0$, we have 
$$g_{Y,\mu}(\gamma,\alpha_0)<\infty\ \Leftrightarrow\ r(\gamma)<\frac 1{\alpha_0}$$
due to Theorem \ref{generalspectraltheorem1} and to Remark \ref{multergodP1P2}. Now let us consider the first case ($J=[0,+\infty] $ and $\delta_0<1/\alpha_0$). If $\gamma \in [0,+\infty]\setminus J_0$, then $r(\gamma)\le \delta_0<1/\alpha_0$ and
$$g_{Y,\mu}(\gamma,\alpha_0)=\sum_{n=0}^{+\infty}\alpha_0^nL_{Y,\mu}^{(n)}(\gamma) 
      \le  \sum_{n=0}^{+\infty}\alpha_0^nC\delta_0^n<\infty.$$
Hence Assertion~$(i)$ is fulfilled. Now, due to Lemma \ref{LEMME0} and to Theorem \ref{thmkellerliverani1}, we have $\lim_{t\rightarrow +\infty}r(t)\le \max(\delta_0,r(\infty))$ and even $\lim_{t\rightarrow +\infty}r(t)=r(\infty)$ if $r(\infty)>\delta_0$. So if $\delta_0 <1/2$, we have
$\eqref{P1}\ \Leftrightarrow\ r(\infty)<1/2$.
\end{proof}
\begin{alem}\label{rnonnul}
Under the assumptions of Theorem \ref{generalspectraltheorem1}, if $ J_0\neq\{0\}$, then we have $r(\gamma)>0$
for every $\gamma\in(0,+\infty)$.
\end{alem}
\begin{proof}
Let $\gamma_0\in J_0$. Due to Lemma
\ref{LEMME0}, $ r (\gamma)\ge r (\gamma_0)> 0$ for every $\gamma\le\gamma_0$.
Next let $\gamma>\gamma_0$ and set $p:=\gamma/\gamma_0>1$. We have
$$0<r (\gamma_0)=r\left(\frac{\gamma}p\right)
   =\lim_{n\rightarrow +\infty}(\pi(P_{\frac\gamma p}^n\mathbf 1_{\mathbb X}))^{\frac 1n} =\lim_{n\rightarrow+\infty
}(\E_\pi[e^{-\frac\gamma p S_n}])^{\frac 1n}$$
due to \eqref{sup-vit} since $\mathbf 1_{\mathbb X}\in \mathcal B$ and since
$\pi\in\mathcal B'$. Moreover, due to the H\"older inequality
we obtain
$$0 <r\left(\gamma_0\right)=\lim_{n\rightarrow+\infty
}(\pi(e^{-\frac\gamma p S_n}))^{\frac 1n}   \le
     \limsup_{n\rightarrow+\infty}(E_\pi[e^{-\gamma S_n}])^{\frac 1{pn}}
               \le (r(\gamma))^{\frac 1p},$$
since $\mathbf 1_{\mathbb X}\in \mathcal B$ and $\pi\in\mathcal B'$,
which implies the positivity of $r(\gamma)$.
\end{proof}

\section{Knudsen gas: Proof of Theorem \ref{pro-Knudsen}\label{Knudsengas}}


In this section, we apply our general results for the Knudsen gas
and more precisely to $P_\gamma$ acting on the usual Lebesgue space  $(\mathbb L^a(\pi),\|\cdot\|_a)$ for some suitable $a\in[1,+\infty)$, where 
\begin{equation} \label{def-La}
\|f\|_a:=\left(\int_{\X}|f(x)|^a\, d\pi(x)\right)^{\frac 1 a}.
\end{equation}
The multiplicative ergodicity follows from Corollary~\ref{cor-th1}
together with the following lemma.
\begin{alem}\label{prop1}
Let $1\le b<a$.
\begin{enumerate}[(i)]
\item For every $\gamma\ge 0$, $r_{ess}(P_{\gamma|\L^{a}(\pi)})\le 1-\alpha$.
\item The function $\gamma\rightarrow P_\gamma $ is continuous from $(0,+\infty]$ to $\mathcal L(\L^{a}(\pi) ,\mathbb L^b(\pi))$.
\item For any $\gamma\in[0,+\infty]$ and any $f\in\L^a(\pi)$, 
$\|P_\gamma f\|_a\le (1-\alpha)\|f\|_a + \alpha\|f\|_1$.
\item For any $\gamma>0$, for any non-null non-negative $f\in\L^a(\pi)$ and every non-null non-negative $g\in\L^{a'}(\pi)$ with $a'= \frac{a}{a-1}$, we have  
$\pi(gP_\gamma f)>0$ and $P_\gamma f >0$.
\item If $r(\gamma)>1-\alpha$, for every $f,g\in\L^a(\pi)$ with $f>0$, $P_\gamma f=r(\gamma)f$ and $P_\gamma g=r(\gamma)g$, then we have $g\in \C\cdot f$.
\item 1 is the only complex number $\lambda$ of modulus 1 such that
$P(h/|h|)=\lambda h/|h|$  in $\L^1(\pi)$ for some $h\in\cB$, $|h|>0$ (modulo $\pi$).
\end{enumerate}
\end{alem}
\begin{proof} $\ $\\
{\it (i)} Observe that $P_\gamma=\alpha \pi(e^{-\gamma\xi}\cdot)+
(1-\alpha)U_\gamma$ with $U_\gamma:=U(e^{-\gamma\xi}\cdot)$.
Since the sum of a Fredholm operator with a compact operator is
Fredholm, we directly obtain
$r_{ess}(P_\gamma)=(1-\alpha)r_{ess}(U_\gamma)\le 1-\alpha$. \\[0.12cm]
{\it (ii)} For every $0\le\gamma,\gamma'<\infty$ and every $f\in\cB$ such that $\|f\|_a=1$, we have
\begin{eqnarray*}
\|P_\gamma f-P_{\gamma'}f\|_b&=& \|P((e^{-\gamma\xi}-e^{-\gamma'\xi})f)\|_b\\
&\le& \|(e^{-\gamma\xi}-e^{-\gamma'\xi})f\|_b\le  \|e^{-\gamma\xi}-e^{-\gamma'\xi}\|_c,
\end{eqnarray*}
where $c$ is such that $\frac 1a+\frac 1c=\frac 1b$.
Hence $\|P_\gamma-P_{\gamma'}\|_{\L^a(\pi),\L^b(\pi)}\le \|e^{-\gamma\xi}-e^{-\gamma'\xi}\|_c $, 
which converges to 0 as $\gamma'$ goes to $\gamma$, by the dominated convergence theorem. In the same way, we prove that $\|P_\gamma-P_{\infty}\|_{\L^a(\pi),\L^b(\pi)}\le \|e^{-\gamma\xi}\|_c $ and hence the continuity of $\gamma\mapsto P_\gamma$
at infinity. \\[0.12cm]
$(iii)$ For every $\gamma\in[0,+\infty]$ and every $f\in\L^a(\pi)$,
$\|P_\gamma f\|_a\le\|Pf\|_a\le (1-\alpha)\|f\|_a + \alpha\|f\|_1$
since $\|Uf\|_a\le\|f\|_a$.
This gives the Doeblin-Fortet inequality. \\[0.12cm]
$(iv)$ For any non-null non-negative $f\in\L^a(\pi)$, we have $P_\gamma f \geq \alpha\pi(e^{-\gamma\xi}f)1_{\X}>0$. The other assertion of (4) is then obvious. \\[0.12cm]
$(v)$ Let $f,g\in\L^a(\pi)$ such that $f>0$, 
$P_\gamma f=r(\gamma)f$ and $P_\gamma g=r(\gamma)g$  in $\L^a(\pi)$.
Set
$\beta := \frac{\pi(e^{-\gamma\xi}g)}{\pi(e^{-\gamma\xi}f)}$ and $h:=g- \beta\, f$.
Then $\pi(e^{-\gamma\xi} h)=0$ and $P_\gamma h = r(\gamma)\, h$, which gives $r(\gamma)\, h = (1-\alpha)\, U(e^{-\gamma\xi} h)$, so that $r(\gamma)\, |h| \leq (1-\alpha)\, U(|h|)$. Since $\pi\, U = \pi$, we obtain: $r(\gamma)\, \pi(|h|) \leq (1-\alpha)\, \pi(|h|)$. Finally we conclude that $\pi(|h|)=0$ because $r(\gamma) > 1-\alpha$ and so $g=\beta f$ in $\L^a(\pi)$. \\[0.12cm]
$(vi)$ 
Let $k\in\L^1(\pi)$ and $\lambda\in \C$ be such that $|\lambda|=1$, 
$|k|\equiv \mathbf 1_\X$ and $P(k)=\lambda k$. Then
$\lambda k=\alpha\pi(k)+(1-\alpha)U(k)$. Taking the modulus, we obtain
$1 \le \alpha|\pi(k)|+(1-\alpha)U(\mathbf 1_\X)\le 1$. By convexity we conclude that
$|\pi(k)|=1$ and that $k$ is constant modulo $\pi$, so that $\lambda=1$.
\end{proof}
\begin{proof}[Proof of the multiplicative ergodicity]
Let $b:=\frac{p}{p-1}$ and $a>b$. From Lemma~\ref{prop1}, $P_\gamma$ satisfies the assumptions of Theorem \ref{generalspectraltheorem1} 
with $\cB_0=\L^a(\pi)$ and $\cB_1=\L^{b}(\pi)$. Moreover $f\mapsto e^{-\gamma\xi}f$ is in $\mathcal L(\L^b(\pi))$. Thus $(S_n)_n$ is multiplicatively ergodic on $\{\gamma>0\, :\, r(\gamma)>1-\alpha\}$ with respect to $\mathbb P_\mu$, provided that $\mu$ defines a continuous linear form on $\L^{b}(\pi)$, that is when $\mu$ is 
absolutely continuous with respect to $\pi$ with density in $\L^{p}(\pi)$. 
\end{proof}
Below, in view of \eqref{P1}, we study the the spectral radius $r(\gamma)$ of $P_\gamma$. First observe that the  nonincreasingness of $r(\cdot)$ follows from Lemma~\ref{LEMME0} since $\L^a(\pi)$ is a Banach lattice. Consequently the set $J_0 := \{\gamma>0\, :\, r(\gamma)>1-\alpha\}$ is an interval with $\min J_0 = 0$ since $r(0)=1$. Next set $h_\gamma:= e^{-\gamma\xi}$ 
for $\gamma\ge 0$ and $ h_\infty:=\mathbf{1}_{\{\xi=0\}}$.
Recall that $P_\gamma f=\alpha\pi(f\, h_\gamma)+
(1-\alpha)U(f\, h_\gamma)$.
We set $\tilde U_\gamma(\cdot):=h_\gamma\, U(\cdot)$.
\begin{alem}\label{rayonspectralKnudsen}
Let $ \gamma\in[0,\infty]$ and $a\in(1,+\infty)$.
Let $\lambda$ be an eigenvalue of $(P_\gamma)_{|\mathbb L^a(\pi)}$
such that $\lambda>(1-\alpha)\rho(\tilde U_\gamma)$. Then
\begin{equation}\label{eqlambda}
\lambda=\alpha\sum_{n\ge 0}\frac{(1-\alpha)^n}{\lambda^n}\pi(\tilde U_\gamma^n(h_\gamma)).
\end{equation}
In particular if $r(\gamma)>1-\alpha$, then $\lambda=r(\gamma)$
satisfies \eqref{eqlambda}.
\end{alem}
\begin{proof}
Let $\gamma\in[0,\infty]$. 
Let $\lambda\in\mathbb C$ and $f\in\mathbb L^a(\pi)$, $f\neq 0$, be such that 
$P_\gamma f=\lambda f$ in $\mathbb L^a(\pi)$, i.e.
$\lambda f=\alpha\pi(f\, h_\gamma)+(1-\alpha)U(f \, h_\gamma)$
that can be rewritten
$$\lambda f\, h_\gamma=\alpha\pi(f\, h_\gamma)\, h_\gamma+(1-\alpha)\tilde U_\gamma
     (f \, h_\gamma).$$
Observe that $\pi(f\, h_\gamma)\ne 0$. Indeed $\pi(f\, h_\gamma) = 0$ would imply
$\lambda f \, h_\gamma=(1-\alpha)\tilde U_\gamma(f \, h_\gamma)$, which contradicts the fact that
$\lambda/(1-\alpha)$ is not in the spectrum of $\tilde U_\gamma$.
Now setting $g:=f\, h_\gamma/\pi(f\, h_\gamma)$, we have
\begin{equation}\label{equationg}
\lambda g=\alpha \, h_\gamma+(1-\alpha)\tilde U_\gamma (g)
\end{equation}
and so
$$\left[id- \frac{1-\alpha}\lambda\tilde U_\gamma\right] (g)=\frac\alpha\lambda \, h_\gamma.$$
Hence
$$ g=\frac\alpha\lambda
  \left[id- \frac{1-\alpha}\lambda\tilde U_\gamma\right]^{-1} (h_\gamma)
   =\frac\alpha\lambda\sum_{n\ge 0}\frac{(1-\alpha)^n}{\lambda^n}
      \tilde U_\gamma^n h_\gamma$$
and so
$$\lambda=\lambda\pi(g)=\alpha\sum_{n\ge 0}\frac{(1-\alpha)^n}{\lambda^n}
      \pi\left(\tilde U_\gamma^n \, h_\gamma\right).$$
\end{proof}

Let $(Z_n)_n$ be a Markov process with transition $U$.
We observe that $\pi(\tilde U_\gamma^n(h_\gamma))=\mathbb E_\pi[e^{-\gamma \sum_{k=0}^nZ_k}]$ if $\gamma\in[0,\infty)$ and $
\pi(\tilde U_\infty^n(h_\infty))= P_\pi[\sum_{k=0}^nZ_k=0]$.
Hence \eqref{eqlambda} can be rewritten
\begin{equation}\label{eqlambdav2}
\lambda=\alpha g_{Z,\pi}(\gamma,(1-\alpha)/\lambda),
\end{equation}
where $ g_{Z,\pi}$ denotes the Laplace-generating function of $(Z_n)_n$ with respect to $\mathbb P_\pi$, i.e.
 $ g_{Z,\pi}(\gamma,x):=\sum_{n=0}^{+\infty}x^n \mathbb E_\pi[e^{-\gamma \sum_{k=0}^nZ_k}]$.
Now \eqref{P1} will follow from the following result, coming from
Corollary \ref{cor-th2}.
\begin{acor}\label{corknuds}
Assume $\alpha>1/2$ and $p>1$. Let $\mu\in\mathbb L^p(\pi)$. Then, with respect to $\mathbb P_\mu$, $\nu$ satisfies
$$\nu=\inf\{\gamma>0\ :\ 2\alpha g_{Z,\pi}(\gamma,2(1-\alpha))<1\}$$
and 
$$\eqref{P1}\quad\Leftrightarrow\quad 2\alpha \sum_{n\ge 0}(2(1-\alpha))^n \P_\pi\left(\sum_{k=0}^{n}Z_k=0\right)<1.$$
In particular \eqref{P1} holds if the random variables $Z_n$ are positive.
\end{acor}
\begin{proof} 
Let $b:=\frac p{p-1}$ and $a>b$.
Let $r(\gamma)$ be the spectral radius of $(P_\gamma)_{\mathbb L^a(\pi)}$.
Let us prove that, for every $\gamma\in[0,+\infty]$, 
$r(\gamma)<1/2\ \Leftrightarrow\
   2\alpha g_{Z,\pi}(\gamma,2(1-\alpha))<1$.

We know that this holds true on $J_0$ due to 
Lemma \ref{rayonspectralKnudsen} and to \eqref{eqlambdav2} (since $g_{Z,\pi}(\gamma,\cdot)$ is increasing).
Now if $\gamma\in[0,+\infty]\setminus J_0$ then $r(\gamma)\le 1-\alpha<1/2$ 
and $2\alpha g_{Z,\pi}(\gamma,2(1-\alpha))\le  2\alpha g_{Z,\pi}(0,2(1-\alpha))<\infty$. We conclude due to Corollary \ref{cor-th1v1}.
\end{proof}
\begin{proof}[Proof of \eqref{P2}]
Assume that $\alpha>1/2$, that $2\alpha \sum_{n\ge 0}(2(1-\alpha))^n \P_\pi\left(\sum_{k=0}^{n}Z_k=0\right)<1$ and $\pi(\xi^\tau)<\infty$ for some $\tau>1$. Let $p>\frac{\tau}{\tau-1}$ and set $a_3:=\frac{p}{p-1}$ (ie.~$1/p+1/a_3 = 1$). Note that $a_3 < \tau$. Let $a_2$ be such that $a_3<a_2<\tau$. Since $\lim_{a\r+\infty} \frac{\tau a}{\tau+ a} = \tau$, we can chose $a_1>a_2$ such that 
$a_2 < \frac{\tau a_1}{\tau+ a_1}$. Next let $a_0 > a_1$. From Lemma~\ref{prop1} we deduce  that Theorem \ref{generalspectraltheorem2} applies with the spaces $\cB_i=\L^{a_i}(\pi)$ for $i=0,1,2,3$. We conclude that $r$ is $C^1$ on $[0,\theta_1)$ with $r'<0$, and so that \eqref{P2} also holds due to Remark \ref{multergodP1P2}, provided that the initial probability measure $\mu$ defines a continuous linear form on $\cB_3=\L^{a_3}(\pi)$, that is (equivalently) when $\mu=h.\pi$ with $h\in L^{p}(\pi)$. 
\end{proof}
%
%
%
%
%
%
%
%
\section{Linear autoregressive model: proof of Theorem \ref{thmAR}}\label{proofAR}
Assume that $\X:=\R$ and $(X_n)_{n\in\N}$ is the linear autoregressive model defined by 
\begin{equation} \label{def-AR}
n\in\N^*, \quad X_n = \alpha X_{n-1} + \vartheta_n\, 
\end{equation}
where $X_0$ is a real-valued random variable, $\alpha\in(-1,1)$, and $(\vartheta_n)_{n\ge 1}$ is a sequence of  i.i.d.~real-valued random variables, independent of $X_0$.  
Assume that $\vartheta_1$ has a positive Lebesgue probability density function on $\X$, say $p(\cdot)$, 
having a moment of order $r_0$ for some  $r_0\ge 1$, that is 
\begin{equation} \label{moment-AR}
\int |x|^{r_0} p(x) dx <\infty. 
\end{equation}
$(X_n)_{n\in\N}$ is a Markov chain with transition kernel 
$$P(x,A)=\int_{\R} \mathbf{1}_A(\alpha x + y) p(y)\, dy = \int_{\R} \mathbf{1}_A( y) p(y-\alpha x)\, dy.$$
Set $V(x) := (1+|x|)^{r_0}$, $x\in\R$. Recall that, under Assumption~(\ref{moment-AR}), $P$ satisfies the following drift condition (see \cite{MeynTweedie09}) 
\begin{equation} \label{inequality-drift}
\forall \delta > |\alpha|^{r_0},\ \exists L\equiv L(\delta) > 0,\quad PV \leq \delta \, V + L\, \mathbf{1}_\X. 
\end{equation}
Moreover it is well-known that $(X_n)_{n\in\N}$ is $V$-geometrically ergodic, see \cite{MeynTweedie09}. 
We also assume 
that, for all $x_0\in \R$, there exist a neighborhood $V_{x_0}$ of $x_0$ and a non-negative Lebesgue-integrable function $q_{x_0}(\cdot)$ such that 
\begin{equation} \label{dom-nu} 
\forall y\in\R,\ \forall v\in V_{x_0},\ p(y+v) \leq q_{x_0}(y). 
\end{equation}
Let $(\cB_V,\|\cdot\|_V)$ be the weighted-supremum Banach space 
\begin{equation} \label{def-BV}
\cB_V := \big\{ \ f : \X\r\C, \text{ measurable }: \|f\|_V  := \sup_{x\in\X} |f(x)| V(x)^{-1} < \infty\ \big\}.
\end{equation}
Let $(\cC_V,\|\cdot\|_V)$ denote the following subspace of $\cB_V$: 
$$\cC_V := \bigg\{ \ f\in\cB_V : \text{ $f$ is continuous and }\ \lim_{|x|\r\infty} \frac{f(x)}{V(x)}\ \text{exists in }\ \C\bigg\},$$
where the symbol $\lim_{|x|\r\infty}$ means that the limits when $x\r\pm\infty$ exist and are equal. 
Note that $V\in\cC_V$ and that $\cC_V$ is a closed subspace of $(\cB_V,\|\cdot\|_V)$.       For every $f\in\cC_V$ we define  
$$\ell_V(f) := \lim_{|x|\r\infty} \frac{f(x)}{V(x)}.$$ 
Let $\cC_{0,V} := \{f\in\cC_V : \ell_V(f)=0\}$. Finally we denote by $(\cC_b,\|\cdot\|_\infty)$ the space of bounded continuous complex-valued functions on $\R$ endowed 
with the supremum norm $\|\cdot\|_\infty$.

We will see below that, for every $\gamma\in[0,+\infty]$, $P_\gamma$ continuously acts on $\cC_V$ (see Lemma~\ref{C-V-0}). We denote by $r(\gamma)$ the spectral radius of $P_\gamma$ on $\cC_V$, that is:  
$$r(\gamma) \equiv r(P_\gamma) := \lim_n\|P_\gamma^n\|_V^{1/n} = \lim_n\|P_\gamma^n V\|_V^{1/n}$$
where $\|\cdot\|_V$ also denotes the operator norm on $\cC_V$. We have $r(0)=r(P)=1$ (see below). 

Recall that $\xi:\X\rightarrow [0,+\infty)$ is a measurable function and that $S_n=\sum_{k=0}^n \xi(X_k)$. Theorem~\ref{proprieteAR} below applied with $\mu=\delta_x$ or $\mu=\pi$ directly provides Theorem~\ref{thmAR}. 
\begin{atheo}\label{proprieteAR}
Assume that the previous assumptions hold.
Let $\mu$ be a probability distribution on $\mathbb R$ belonging to $\cC_V^{\, *}$, namely satisfying $\mu(V) < \infty$.
Assume moreover that $\xi$ is continuous, coercive, that $p$ is continuous, and that $\sup_{\mathbb R}\xi/V<\infty$. Then
\begin{enumerate}
\item $\rho_Y=r$ on $[0,+\infty)$ and
$(S_n)_n$ is multiplicatively ergodic on $[0,+\infty)$ with respect to $\P_\mu$. 
\item If moreover the Lebesgue measure of the set $[\xi=0]$ is zero,
then $\lim_{\gamma\rightarrow +\infty}r(\gamma)=0$. Hence \eqref{P1} holds true under $\mathbb P_\mu$.
\item Moreover, if  there exists $\tau>0$ such that $\sup_{\R}\xi^{1+\tau}/V<\infty$, then $\gamma\mapsto   r(\gamma)$ admits a negative derivative on $[0,+\infty)$. Hence \eqref{P2} holds also true  under $\mathbb P_\mu$.
\end{enumerate}
\end{atheo}
\subsection{Quasi-compactness of $P_\gamma$}
We start this section with the following useful lemma.
\begin{alem} \label{C-V-0}
Assume that Assumption~(\ref{moment-AR}) holds, that $p$ is continuous, that $\xi$ is continuous and coercive. Then, for every $\gamma\in[0,+\infty)$, $P_\gamma$ continuously acts on $\mathcal C_V$. Moreover, for every $\gamma\in(0,+\infty)$, we have $P_\gamma(\cC_V) \subset \cC_{0,V}$ and $P_\gamma$ is compact from $\cC_b$ into $\cC_V$. 
\end{alem}
\begin{proof}{}
Let $\gamma\in[0,+\infty)$. From (\ref{inequality-drift}) it easily follows that $P_\gamma V \leq PV \leq (\delta+L)V$, so that $P_\gamma$ continuously acts on $\cB_V$. Now let $f\in\cC_V$. Then  
$$\forall x\in\R,\quad \frac{(P_\gamma f)(x)}{V(x)} = \int_\R \chi(x,y)\, dy\qquad \text{with } \ \chi(x,y) := e^{-\gamma \xi(\alpha x+y)}\, \frac{f(\alpha x+y)}{V(x)}\, p(y).$$
We have for every $(x,y)\in\R^2$ 
$$|\chi(x,y)| \, \leq\,  \|f\|_V \left(\frac{1+|x|+|y|}{1+|x|}\right)^{r_0} p(y) \, \leq\,  \|f\|_V \big(1+|y|\big)^{r_0} p(y)$$
Since $\chi(\cdot,y)$ is continuous for every $y\in\R$, we deduce from (\ref{moment-AR}) and Lebesgue's theorem that the function $P_\gamma f/V$ is continuous on $\R$, thus so is $P_\gamma f$. This proves that $P_\gamma(\cC_V) \subset \cC_V$, thus $P_\gamma$ continuously acts on $\cC_V$. 

Now let us consider $\gamma\in(0,+\infty)$. Since 
$$|\chi(x,y)|\leq \|f\|_V\, e^{-\gamma \xi(\alpha x+y)}\, \big(1+|y|\big)^{r_0} p(y)$$ 
and $\lim_{|x|\r +\infty} e^{-\gamma \xi(\alpha x+y)} = 0$, it follows again from Lebesgue's theorem that 
$$\lim_{|x|\r +\infty} \frac{(P_\gamma f)(x)}{V(x)} = 0,$$
thus $P_\gamma f \in \cC_{0,V}$. 

To prove the last assertion, observe that, since $p$ is continuous, the image by $P$ of 
the unit ball $\{f\in\cC_b : \|f\|_\infty \leq 1\}$ in $\cC_b$ is equicontinuous from Scheff\'e's lemma. Then $P$ 
is compact from $\cC_b$ into $\cC_V$ from Ascoli's theorem and from $\lim_{|x|\r\infty} V(x) = +\infty$. Next, for every $\gamma>0$, we have $P_\gamma = P\circ M_\gamma$ with $M_\gamma f = e^{-\gamma\xi} f$. Thus  $P_\gamma$ is compact from $\cC_b$ into $\cC_V$ since $M_\gamma$ is a bounded linear operator on $\cC_b$ and $P$ is compact from $\cC_b$ into $\cC_V$. 
\end{proof}
%
%


Here we use the duality arguments of \cite[prop.~5.4]{HerLed14} to prove the quasi-compactness of $P_\gamma$ on $\mathcal C_V$. The topological dual spaces of $\cC_V$ and $\cC_b$ are denoted by $(\cC_V^*,\|\cdot\|_V)$ and $(\cC_b^*,\|\cdot\|_\infty)$ respectively (for the sake of simplicity we use the same notation for the dual norms). For any $\gamma>0$, we denote by $P^*_\gamma$ the adjoint operator of $P_\gamma$ on $\cC_V$.  Note that each $P_\gamma^*$ is a contraction with respect to the dual norm $\|\cdot\|_\infty$  because so is $P_\gamma$ on $\cC_b$.

In the sequel, $\delta > |\alpha|^{r_0}$ is fixed, as well as the associated constant $L\equiv L(\delta)$ in (\ref{inequality-drift}). 
Iterating Inequality (\ref{inequality-drift}) proves that $P$ is power-bounded on $\cC_V$ (i.e.~$\sup_{n\geq 1}\|P^nV\|_V < \infty$), thus $r(0)=r(P)=1$ since $P$ is Markov. Moreover (\ref{inequality-drift}) rewrites as the following (dual) Doeblin-Fortet inequality (see the proof in \cite[p.~190]{FerHerLed13}):
\begin{equation}\label{DFgamma=0} 
\forall f^*\in\cC_V^*,\quad \|P^* f^*\|_V \leq \delta  \, \|f^*\|_V + L\, \|f^*\|_\infty.
\end{equation}
Since $P$ is compact from $\cC_b$ into $\cC_V$ (Lemma~\ref{C-V-0}), so is $P^*$ from $\cC_V^*$ into $\cC_b^*$. Then we deduce from \cite{Hen93} that, under Assumption~(\ref{moment-AR}), $P$ is a quasi-compact operator on $\cC_V$ and its essential spectral radius  $r_{ess}(P)$ satisfies the following bound (see also \cite[Sect.~8]{Wu04}): 
\begin{equation} \label{r-ess-P}
r_{ess}(P) \leq \delta
\end{equation}
The next lemma extends Inequality~(\ref{inequality-drift}) to the operators $P_\gamma$. 
\begin{alem} \label{lem-D-F-P-gamma}
Assume that Assumption~(\ref{moment-AR}) holds true and that $\xi$ is coercive. 
Then, for every $\gamma >0$ and for every $\beta>0$, there exists a positive constant $L_{\beta}$ such that 
\begin{equation} \label{drift-P-gamma}
P_\gamma V \leq e^{-\gamma\beta}\, \delta \, V + L_{\beta}\, \mathbf{1}_\X 
\end{equation}
Moreover
\begin{equation} \label{drift-P-gamma-infty}
P_{\infty} V \leq \left(\sup_{[\xi =0]} V\right) \mathbf{1}_\X.
\end{equation}
\end{alem}
\begin{proof}{}
We have for every $\gamma >0$ and for every $\beta>0$ 
\begin{eqnarray*}
P_\gamma V &=& P(e^{-\gamma\xi} V) = P\big(e^{-\gamma\xi} \mathbf{1}_{[\xi > \beta]}V\big) + P\big(e^{-\gamma\xi} \mathbf{1}_{[\xi \leq \beta]}V\big)  \\ 
&\leq& e^{-\gamma\beta} \big(\delta \, V + L\, \mathbf{1}_\X\big) + \int_{[\xi \leq \beta]} V(y) P(\cdot,dy) \qquad \qquad \text{(from (\ref{inequality-drift}))}  \\ 
&\leq& e^{-\gamma\beta}\, \delta \, V + \big(L + \sup_{[\xi \leq \beta]} V\big)\mathbf{1}_\X 
\end{eqnarray*}
from which we deduce the first desired statement. 
For $P_{\infty}$, we have
$$
P_{\infty} V=P(\mathbf 1_{\{\xi=0\}}V) \leq 
\left(\sup_{[\xi =0]} V\right) P(\mathbf 1_\X)=\left(\sup_{[\xi =0]} V\right)\mathbf 1_\X.
$$
\end{proof}

\begin{acor} \label{cor-QC-gamma} 
Assume that Assumption~(\ref{moment-AR}) holds true and that $\xi$ is coercive. Then, for every $\gamma\in(0,+\infty]$, $P_\gamma$ is a quasi-compact operator on $\cC_V$ and its essential spectral radius $r_{ess}(P_\gamma)$ is zero. 
\end{acor}
\begin{proof}{}
Consider any $\gamma>0$. For any $\varepsilon > 0$, choose $\beta = \beta(\gamma,\varepsilon)>0$ such that $e^{-\gamma\beta}\, \delta < \varepsilon$. Then we deduce from Lemma~\ref{lem-D-F-P-gamma} that $P_\gamma V \leq \varepsilon\, V + D\, \mathbf{1}_\X$, where $D\equiv D(L,\gamma,\varepsilon)$ is a positive constant. This inequality rewrites as (see the proof in \cite[p.~190]{FerHerLed13}): 
\begin{equation}\label{DFgamma>0}
\forall f^*\in\cC^*_V,\quad \|P^*_\gamma f^*\|_V \leq \varepsilon \, \|f^*\|_V + D\, \|f^*\|_\infty.
\end{equation}
Since $P^*_\gamma$ is compact from $\cC^*_V$ into $\cC^*_b$ (Lemma~\ref{C-V-0}), we deduce from \cite{Hen93} that $P_\gamma$ is quasi-compact on $\cC_V$ with $r_{ess}(P_\gamma) \leq \varepsilon$. We obtain $r_{ess}(P_\gamma) = 0$ because $\varepsilon$ is arbitrary.  
\end{proof}
With the usual convention $V^0:=1$, we have the identification $\cC_{V^0}=\cC_b$.
\begin{arem}\label{cor-QC-gammabis}
Let $\varepsilon>0$ , $0 \le a\le a+b\le 1$.
Observe that Corollary \ref{cor-QC-gamma} holds also if we replace $V$ by $V^{a+b}$ (since $\vartheta_1$ admits a moment of order $r_0(a+b)$). Moreover notice that 
\eqref{DFgamma>0}  with $V^{a+b}$
instead of $V$ directly gives
\begin{equation}\label{DFgamma>0bis}
\exists D_{\varepsilon,a+b}>0,\quad\forall f^*\in\cC^*_{V^{a+b}},\quad \|P^*_\gamma f^*\|_{V^{a+b}} \leq \varepsilon \, \|f^*\|_{V^{a+b}} + D_{\varepsilon,a+b}\, \|f^*\|_{V^a}
\end{equation}
since $\Vert f^*\Vert_\infty\le\Vert f^*\Vert_{V^a}$.
\end{arem}

\subsection{Continuity of the function $\gamma\mapsto r(\gamma)$}

\begin{apro} \label{pro-cont-r-gamma}
Assume that Assumption~(\ref{moment-AR}) holds true, that $\xi$ is coercive  and finally that the function $\xi/V$ is bounded on $\R$. Then the function $\gamma\mapsto r(\gamma)$ is continuous on $[0,+\infty]$.
\end{apro}
The continuity of $\gamma\mapsto r(\gamma)$ at some $\gamma_0\in[0,+\infty]$ directly follows from the continuity of $\gamma\mapsto P_\gamma$ from $[0,+\infty]$ to $\cL(\cC_b,\cC_V)$ due to Theorem \ref{thmkellerliverani1} (applied with any $\delta_0\in(0,1)$ if $\gamma_0\ne 0$ and with any $\delta_0>|\alpha|^{r_0}$ if $\gamma_0=0$), and due to \eqref{DFgamma=0}, to \eqref{DFgamma>0} and to Corollary \ref{cor-QC-gamma}. Hence Proposition \ref{pro-cont-r-gamma} comes from the following lemmas.
\begin{alem} \label{lem-cont-P-gamma}
Let $0\leq a < a+b \leq 1$. 
Assume that $\xi \leq cV$ for some positive constant $c$. Then the following operator-norm inequality holds for every $(\gamma,\gamma')\in[0,+\infty)^2$
$$
\|P_\gamma - P_{\gamma'}\|_{{\cC}_{V^a},{\cC}_{V^{a+b}}}\ := \sup_{f\in{\cC}_{V^a},\|f\|_{V^{a}}\leq 1}\|P_\gamma f - P_{\gamma'} f\|_{V^{a+b}} \ \ \leq \ (c|\gamma - \gamma'|)^b\Vert P\Vert_{V^{a+b}} \, .$$
\end{alem}
\begin{proof}{}
Let $(\gamma,\gamma')\in[0,+\infty)^2$. For all $(u,v)\in[0,+\infty)^2$, we have $|e^{-u} - e^{-v}| \leq |e^{-u} - e^{-v}|^b \leq |u-v|^b$ from Taylor's inequality. Thus we obtain for any $f\in\cC_{V^a}$ 
\begin{eqnarray*}
\big|(P_\gamma f)(x) - (P_{\gamma'} f)(x)\big| &\leq&\Vert f\Vert_{V^a}  \int_\R \big|e^{-\gamma\xi(y)} - e^{-\gamma'\xi(y)}\big| (V(y))^a p(y-\alpha x)\, dy \\ 
&\leq& \Vert f\Vert_{V^a} (c\, |\gamma - \gamma'|)^b \int_\R (V(y))^{a+b}\,  p(y-\alpha x)\, dy \\ 
&\leq& \Vert f\Vert_{V^a} (c\, |\gamma - \gamma'|)^b  P V^{a+b}(x),
\end{eqnarray*}
from which we deduce the desired inequality. 
\end{proof}
\begin{alem} \label{prop-1-AR} 
Assume that Assumptions~(\ref{moment-AR}) and  (\ref{dom-nu}) hold true, that $\xi$ is coercive. Then 
$$\|P_\gamma- P_{\infty}\|_{{\cC}_b,{\cC}_V} := \sup_{f\in{\cC}_b,\|f\|_\infty\leq 1}\|P_\gamma f- P_{\infty}f\|_V \ \longrightarrow 0 
\quad \text{when}\  \gamma\r+\infty.$$ 
\end{alem}
\begin{proof}{}
Let $\varepsilon >0$. Let $f\in\cC_b$ be such that $\|f\|_\infty \leq 1$. From $|P_\gamma f| \leq P\mathbf{1}_\X = \mathbf{1}_\X$ it follows that there exists $A\equiv A(\varepsilon)$ such that : 
\begin{equation} \label{inequality-1}
|x| > A\ \Rightarrow\ \forall \gamma\in(0,+\infty),\ \frac{|(P_\gamma f)(x)|}{V(x)} \leq \varepsilon.
\end{equation} 
Next we deduce from Assumption~(\ref{dom-nu}) and a usual compactness argument ($[-A,A]$ is compact) that there exists a Lebesgue-integrable function $q\equiv q_A$ such that 
$$\forall v\in [-A,A],\ \forall y\in\R,\quad p(y+v) \leq q(y).$$
Consequently we obtain for any $\beta >0$ and $x\in\R$ such that $|x|\leq A$  
\begin{eqnarray*}
\big|(P_\gamma f-P_{\infty} f)(x)\big| &\le& e^{-\gamma\beta} \int_{[\xi > \beta]} p(y-\alpha x)\, dy + \int_{[0<\xi \leq \beta]} p(y-\alpha x)\, dy  \\ 
&\leq& e^{-\gamma\beta}  + \int_{[0<\xi \leq \beta]} q(y)\, dy.
\end{eqnarray*}
Since $\int_{[0<\xi \leq \beta]} q(y)\, dy \r 0$ when $\beta\r 0$, there exists $\beta_0\equiv \beta_0(\varepsilon)>0$ such that 
$$\int_{[0<\xi \leq \beta_0]} q(y)\, dy \leq \frac{\varepsilon}{2}.$$
Finally let $\gamma_0 \equiv \gamma_0(\varepsilon)>0$ be such that : $\forall \gamma > \gamma_0,\ e^{-\gamma\beta_0} \leq \varepsilon/2$. Then  
\begin{equation} \label{inequality-2}
|x| \leq A\ \Rightarrow\ \forall \gamma\in(\gamma_0,\infty),\ \frac{|(P_\gamma f
        - P_{\infty} f)(x)|}{V(x)} \leq  |(P_\gamma f)(x) - (P_{\infty}f)(x)| \leq  \varepsilon.
\end{equation}
Inequalities (\ref{inequality-1}) and (\ref{inequality-2}) provides the desired statement. 
\end{proof}
%


\subsection{Proof of the two first points of Theorem \ref{proprieteAR}}
%

Let $\theta_1:=\sup\{\gamma>0\, :\, r(\gamma)>0\}$.
Since $r$ is continuous at $0$ and $r(0)=1$, we observe that $\theta_1>0$.
Let us prove that the assumptions of Theorem \ref{generalspectraltheorem1}
hold true on $J=(0,\theta_1)$ with $\cB_0 :=\mathcal \cC_{V^a}$ for some (any) $a\in(0,1)$ and $\cB_1=\cC_V$. Note that $\mathbf 1_{\X}\in\cB_0$.
The fact that $(P_\gamma)_\gamma$ satisfies the conditions of 
Hypothesis~\ref{hypKL}*  on $J$ with $\cB_0 = \cC_b$ and $\cB_1 = \cC_V$  comes from \eqref{DFgamma>0} and from Lemmas \ref{lem-cont-P-gamma} and \ref{prop-1-AR}. Moreover we prove below that Hypothesis \ref{hypcompl} holds with respect to $(J,\cB_1)=(J,\mathcal\cC_{V})$. Since $f\mapsto e^{-\gamma\xi}f$ is in $\cL(\cC_V)$, we then deduce from Corollary~\ref{cor-th1} that $(S_n)_n$ is multiplicatively ergodic 
on $(0,\theta_1)$ and so $\rho_Y(\gamma)=r(\gamma)>0$
on $(0,\theta_1)$. Moreover, since $\theta_1>0$, it follows from Lemma \ref{rnonnul} 
 that $\theta_1=+\infty$. We have proved Assertion~$(1)$ of 
Theorem \ref{proprieteAR}.
For Assertion~$(2)$, observe that $\Leb(\xi=0)=0$
implies that $P_{\infty}=0$, in particular we have $r(+\infty)=0$. Then Theorem~\ref{thmkellerliverani1} gives $\lim_{\gamma\rightarrow +\infty}r(\gamma)=r(+\infty)=0$. 
Consequently $\nu$ is finite and satisfies \eqref{P1bis}, and so \eqref{P1}, with respect to $\P_\mu$, provided that $\mu$ is a probability distribution $\mu$ belonging to $\mathcal C_V^*$.

Recall that the previous proof shows that $r(\gamma)>0$ for every $\gamma \geq 0$. It remains to establish that Hypothesis \ref{hypcompl} holds with respect to $(J,\cB_1)=(J,\mathcal\cC_{V})$. This is provided by Remark~\ref{rqe-dec-fl} and Lemmas~\ref{lem-val-prop-simple}-\ref{lem-val-prop-simple2} below. 
\begin{alem} \label{dec-fl} 
For any non-null $e^*\in \cC_V^*$, $e^*\geq 0$, there exists a nonnegative measure $\mu\equiv \mu_{e^*}$ on $(\R,\cX)$ such that 
\begin{equation} \label{dec-fl-exp}
\forall f\in\cC_V,\quad e^*(f) = \mu\left(\frac{f}{V} - \ell_V(f)\, \mathbf{1}_\R\right) + e^*(V)\, \ell_V(f).
\end{equation}
\end{alem}
\begin{arem}\label{rqe-dec-fl}
Due to Lemma~\ref{dec-fl}, the first condition of Hypothesis~\ref{hypcompl} is fulfilled with $J=[0,+\infty)$ and $\cB=\cC_V$. Indeed, let $\gamma\in [0,+\infty)$ and let $\phi\in\cC_V$ be non-null and non-negative. Then, we have $P_\gamma \phi > 0$ everywhere from the definition of $P$ and the strict positivity of the function $p(\cdot)$. Moreover, if $\psi\in\cB^*\cap\ker(P_\gamma^*-r(\gamma)I)$ is non-null and non-negative, then we have $\psi(P_\gamma\phi)>0$. Indeed this property holds for $\gamma=0$ since we know that $\psi = c\, \pi$ for some $c>0$ and that $P_\gamma \phi > 0$ everywhere. Now let $\gamma>0$. First observe that $\psi\neq c\, \ell_V$ for every $c\in\C$ because $r(\gamma) > 0$ and $P_\gamma^*(\ell_V) = 0$ from Lemma~\ref{C-V-0}. Second note that $\mu=0$ in (\ref{dec-fl-exp}) implies that $e^*=e^*(V)\, \ell_V$. Thus the nonnegative measure $\mu\equiv \mu_{\psi}$ associated with $\psi$ in (\ref{dec-fl-exp}) is non-null. Since $\ell_V(P_\gamma\phi)=0$ from Lemma~\ref{C-V-0}, we deduce from (\ref{dec-fl-exp}) (applied with $e^*=\psi$) and from $P_\gamma \phi > 0$ that $\psi(P_\gamma\phi) = \mu(P_\gamma\phi/V)>0$.
\end{arem}
\begin{proof}[Proof of Lemma~\ref{dec-fl}]
Let $(\cC,\|\cdot\|)$ denote the following space 
$$\cC := \bigg\{ \ g : \R\r\C\  \text{ continuous }: \|g\| := \sup_{x\in\R} |g(x)| < \infty\ \text{ and }\ \lim_{|x|\r\infty} g(x)\ \text{exists in}\ \C\bigg\}.$$ 
For every $g\in\cC$, we set: $\ell(g) := \lim_{|x|\r\infty} g(x)$. We denote by $\cC^*$ the topological dual space of $\cC$. Let $e^*\in \cC_V^*$, $e^*\geq 0$, and let $\widetilde e^*\in\cC^*$ be defined by: 
$$\forall g\in\cC,\quad \widetilde e^*(g) := e^*(gV).$$ 
Next let $\widetilde e^*_0$ be the restriction of $\widetilde e^*$ to $\cC_0:=\{g\in\cC : \ell(g)=0\}$. From the Riesz representation theorem, there exists a unique positive measure $\mu$ on $(\R,\cX)$ such that 
$$\forall g\in\cC_0,\quad \widetilde e^*_0(g) = \mu(g) := \int_\R g\, d\mu.$$ 
Then, writing $g = (g-\ell(g)\, \mathbf{1}_\R)+ \ell(g)\, \mathbf{1}_\R$ for any $g\in\cC$, we obtain that  
$$\widetilde e^*(g) = \mu\big(g-\ell(g)\, \mathbf{1}_\R\big) +  \widetilde e^*(\mathbf{1}_\R)\, \ell(g).$$
We conclude by observing that, for any $f\in\cC_V$, we have $e^*(f) = \widetilde e^*(f/V)$. 
\end{proof}

\begin{alem} \label{lem-val-prop-simple} 
If $f,g\in \cC_V$ are such that
$P_\gamma f=r(\gamma)f$ and $P_\gamma g=r(\gamma)g$ with $f>0$, then $g\in \C\cdot f$. 
\end{alem}
\begin{proof}{}
Let $f,g\in\ker(P_\gamma - r(\gamma) I)$ with $f>0$.
Let $\beta\in\C$ be such that $h:=g - \beta f$ vanishes at $0$. Since $h\in\ker(P_\gamma - r(\gamma) I)$ we deduce from Proposition~\ref{firstorder}  that $P_\gamma |h| = r(\gamma) |h|$. Then $|h|(0)=0$, the positivity of $p(\cdot)$ and finally the continuity of $|h|$ show that  $h=0$. 
\end{proof}
\begin{alem} \label{lem-val-prop-simple2} 
Let $h\in\cC_V$ with $|h|>0$ and $\lambda\in\mathbb C$ be such that
$|\lambda|=1$ and $P\frac{h}{|h|}=\lambda \frac{h}{|h|}$ in $\mathbb L^1(\pi)$. Then $\lambda=1$.
\end{alem}
\begin{proof}{}
Observe that $\frac h{|h|}$ is in $\mathcal C_b$ so in $\cB_V$.
But it is known from \cite{MeynTweedie09} that $(X_n)_n$
is $V$-geometrically ergodic, so $\lambda$=1.
\end{proof}
\subsection{Proof of Part~(3) of Theorem \ref{proprieteAR}}
We assume now that $\xi\in\cB_{V^{\frac 1{1+\tau}}}$ for some $\tau>0$
and that $[\xi=0]$ has Lebesgue measure 0.

Let $0<a_0<a_1<a_1+\frac{1}{1+\tau}<a_2<a_3=1$.
Let us prove that the additional assumptions of Theorem \ref{generalspectraltheorem2}
hold true with $\cB_i:=\mathcal \cC_{V^{a_i}}$
for $i\in\{0,1,2,3\}$. Let $i\in\{0,1,2\}$.
The fact that $(P_\gamma)_\gamma$ satisfies the conditions of Hypothesis \ref{hypKL}* on $(J,\cB_{i},\cB_{i+1})$  comes from Lemma \ref{lem-cont-P-gamma} and 
Remark \ref{cor-QC-gammabis}.
The fact that Hypothesis \ref{hypcompl} is satisfied on $\cB_{i+1}$
comes from Remark \ref{rqe-dec-fl} and from Lemmas \ref{lem-val-prop-simple} and \ref{lem-val-prop-simple2} applied with $V^{a_{i+1}}$ (in place of $V$). 
Observe that 
$$\|\xi f\|_{\cB_2}=\sup\frac{\|\xi f\|}{V^{a_2}}
      \le\sup\frac {\|\xi \|}{V^{\frac 1{1+\tau}}}\, \sup\frac{\| f\|}{V^{a_1}}
      \le \|f\|_{\cB_1}\sup \frac{\|\xi \|}{V^{\frac 1{1+\tau}}}.$$
Hence we have proved that $f\mapsto \xi f$ is in $\cL(\cB_1,\cB_2)$.
The fact that $\gamma\mapsto P_\gamma$ is $C^1$ from $(0,+\infty)$
to $\cL(\cB_1,\cB_2)$ and that $P'_\gamma:=P_\gamma(-\xi f)$ 
comes from the proof of \cite[Lemma 10.4]{HerPen10}.
We conclude as explained after Theorem \ref{generalspectraltheorem2}. 

\appendix
\section{Operator techniques} \label{proofoperator}

We use the notations of Section \ref{nota}.
\subsection{Decrease of $r$}
The non-increasingness of $r$ was studied in Lemma \ref{LEMME0}.
The next result gives a way to prove that $r'\ne 0$ and so the decrease of $\gamma\mapsto r(\gamma)$. 
\begin{apro}\label{deriveenegative} 
Let $J=(a,b)\subset[0,+\infty)$ and let $ \cB_1\hookrightarrow \cB_2$ be two Banach spaces such that $P_\gamma\in\cL(\cB_1)\cap\cL(\cB_2)$ for every $\gamma\in J$. Assume that
$f\mapsto \xi f\in\cL(\cB_1,\cB_2)$ and that,
for every $\gamma\in J$, there exist $\phi_\gamma\in\cB_1$ and $\pi_\gamma\in \cB^*_2$ such that
$P_\gamma\phi_\gamma=r(\gamma)\phi_\gamma$
and $P^*_\gamma\pi_\gamma=r(\gamma)\pi_\gamma$ (where $P^*_\gamma$ is the dual operator of $P_\gamma$).
Let $\gamma_0$ be a point of $J$ at which the functions $\gamma\mapsto P_\gamma$ from $J$ to $\cL(\cB_1,\cB_2)$ and  $\gamma\mapsto
r(\gamma)$ from $J$ to $\C$ are differentiable with
respective derivatives $f\mapsto P_{\gamma_0}(-\xi f)$ and $r'(\gamma_0)$.
We assume moreover that, at $\gamma_0$, $\gamma\mapsto\phi_\gamma$ is continuous from $J$ to $\cB_1$ and differentiable from $J$ to $\cB_2$
with derivative $\phi'_{\gamma_0}$. 

If $r(\gamma_0)\neq0$ and $r'(\gamma_0)=0$ then  $\pi_{\gamma_0}(\xi\phi_{\gamma_0})=0$.
\end{apro}
\begin{proof}
We have $P_\gamma\phi_\gamma = r(\gamma)\phi_\gamma$ in $\cB_2$.
We derive this formula at $\gamma_0$ by writing $P_\gamma\phi_\gamma-
P_{\gamma_0}\phi_{\gamma_0}=P_{\gamma_0}(\phi_\gamma-\phi_{\gamma_0})+
(P_\gamma-P_{\gamma_0})(\phi_{\gamma})$. Using the fact that
$r'(\gamma_0) = 0$, we obtain that
$$  P_{\gamma_0}(\phi'_{\gamma_0})+P_{\gamma_0}( -\xi \phi_{\gamma_0})
=r(\gamma_0)\phi'_{\gamma_0}\ \ \ \mbox{in}\ \cB_2.$$
Composing by $\pi_{\gamma_0}$, we obtain
$0 = \pi_{\gamma_0} P_{\gamma_0}\big(\xi \phi_{\gamma_0}\big) =  r(\gamma_0)\pi_{\gamma_0} \big(\xi \phi_{\gamma_0}\big)$, thus $\pi_{\gamma_0} \big(\xi \phi_{\gamma_0}\big) = 0$.
\end{proof}
\subsection{Proof of Theorems \ref{generalspectraltheorem1} and \ref{generalspectraltheorem2}} \label{ap-proof-general-th}
Let us state, more precisely than in Theorem \ref{thmkellerliverani1}, the Keller-Liverani perturbation theorem.  
\begin{atheo}[Keller-Liverani Perturbation Theorem \cite{KelLiv99,Liv03,Ferre}]
\label{thmkellerliverani}
Let $(\mathcal X_0,\|\cdot\|_{\cX_0})$ be a Banach space and $(\mathcal X_1,\|\cdot\|_{\mathcal X_1})$ be a normed space such that $\mathcal X_0\hookrightarrow \mathcal X_1$. Let $J\subset[-\infty,+\infty]$ be an interval and let $(Q(t))_{t\in J}$ be a family of operators. We assume that 
\begin{itemize}
\item For every $t\in J$, $Q(t)\in \mathcal L(\mathcal X_0)\cap\mathcal L(\mathcal X_1)$,
\item $t\mapsto Q(t)$ is a continuous map from $J$ in $\mathcal L( \mathcal X_0,\mathcal X_1)$,  
\item There exist $\delta_0>0$, $c_0,M_0>0$ such that for every $t\in J$
$$\forall f\in\mathcal X_0,\ \forall n\in \mathbb Z_+,
    \quad
       \|(Q(t))^n f\|_{\mathcal X_0}\le c_0(\delta_0^n\|f\|_{\cX_0}
           + M_0^n\Vert  f\Vert_{\cX_1}).$$
\end{itemize}
Let $t_0\in J$. Then, for every $\varepsilon>0$ and every $\delta>\delta_0$, there exists $I_0\subset J$ containing $t_0$ such that
$$\sup_{t\in I_0,\, z\in\mathcal D(\delta,\varepsilon)}\|(zI-Q(t))^{-1}\|_{\cX_0}<\infty,$$
with $\mathcal D(\delta,\varepsilon):=\{z\in \mathbb C,\ 
       d(z,\sigma(Q(t_0)_{|\cX_0}))>\varepsilon,\ |z|>\delta\}$. 
					
Furthermore the map $t\mapsto(zI-Q(t))^{-1}$ from $J$ to 
$\mathcal L(\mathcal X_0,\mathcal X_1)$ is continuous at $t_0$ in a uniform way with respect to $z\in\mathcal D(\delta,\varepsilon)$, i.e.
$$\lim_{t\rightarrow t_0,\, t\in J} \sup\left\{\|(zI-Q(t))^{-1}-(zI-Q(t_0))^{-1}\|_{\cX_0,\cX_1};\ z\in \mathcal D(\delta,\varepsilon)\right\}=0.$$
In particular, $\limsup_{t\rightarrow t_0}
r((Q(t))_{|\mathcal X_0})\le \max(\delta_0,r((Q(t_0))_{|\mathcal X_0}))$.
Moreover the map
$t\mapsto r((Q(t))_{|\mathcal X_0})$ is continuous on $\{t\in J : r((Q(t))_{|\cX_0}) > \delta_0\ge r_{ess}((Q(t))_{|\cX_0})\}$.
\end{atheo}

Let $\mathcal B$ be a nonnull complex Banach lattice of functions 
$f:\X\rightarrow \C$ (or of classes of such functions
modulo a nonnegative nonnull measure $\mathbf m$).
If $f\in\mathcal B$ is a class of functions, we say that it is nonnegative
resp. positive if one of its representant is so and we say that it is nonnull
if the null function is not one of its representant.
We say that $\psi\in\mathcal B^*$ is nonegative 
if for every nonnegative $f\in\mathcal B$, $\psi(f)\ge 0$ and that 
$\psi\in\mathcal B^*$ is positive 
if for every nonnegative nonnull $f\in\mathcal B$, $\psi(f)> 0$.
\begin{apro}[First order of the spectral radius] \label{firstorder} 
Let $\mathcal B$ be a non null complex Banach lattice of functions 
 $f:\X\rightarrow \C$ (or of classes of such functions
modulo some nonnegative nonnull measure $\mathbf m$). Let $Q$ be a (nonnull) nonnegative quasicompact operator on $\mathcal B$ such that $r(Q)\ne 0$ and such that
for every nonnull nonnegative $f\in\mathcal B$ and for every nonnull nonnegative $\psi\in\mathcal B ^*\cap\ker(Q^*-r(Q)I)$, we have $Qf>0$ (modulo $\mathbf m$) and $\psi(Qf)> 0$.
Then
\begin{itemize}
\item $r(Q)$ is a first order pole of $Q$, and there exists a positive $\phi\in\cB$ and a  positive $\psi\in\cB ^*$ such that
$$\psi(\phi)=1, \qquad Q \phi = r(Q) \phi \qquad\text{and}\qquad Q^*\psi = r(Q) \psi.$$
\item Let $\lambda\in\C$ and $h\in\cB$ such that $|\lambda|=r(Q)$ and $Q h = \lambda h$. Then $Q |h| = r(Q) |h|$ in $\cB$.
\item If moreover $Q$ is if the form $Q=P(e^{-\gamma\xi}\cdot)$ where $P$ is the operator associated to
a Markov kernel, if $1_{\X}\in \cB\hookrightarrow \L^1(\pi)$, if $\ker(Q-r(Q)I)=\C\cdot \phi$ and  
if 1 is the only complex number $\lambda$ of modulus 1 such that
 $P(h/|h|)=\lambda h/|h|$ in $\L^1(\pi)$ for some $h\in\cB$ with $|h|>0$, then $r(Q)$ is the only eigenvalue of modulus 
$r(Q)$ of $Q$.
\end{itemize}
\end{apro}
\begin{proof}
The fact that $r(Q)$ is a finite pole of $Q$ is classical for a nonnegative quasi-compact operator $Q$ on a Banach lattice. Let us just remember the main arguments. From quasi-compactness we know that there exists a finite pole $\lambda\in\sigma(Q)$ such that $|\lambda|=r(Q)$. Thus, setting $\lambda_n:=\lambda(1+1/n)$ for any $n\geq 1$, we deduce from $\lambda\in\sigma(Q)$ that 
$$\lim_n \|(\lambda_nI - Q)^{-1}\|_\cB = +\infty.$$
Since $\cB$ is a Banach lattice, we deduce from the Banach-Steinhaus theorem that there exists a nonnegative and nonnull element $f\in\cB$ such that 
$$\lim_n \|(\lambda_nI - Q)^{-1}f\|_\cB = +\infty.$$
Next define $r_n :=r(Q)(1 +1/n)$ and observe that 
$$\big|(\lambda_nI - Q)^{-1}f\big| = \big|\sum_{k\geq 0}  \lambda_n^{-(k+1)}\, Q^{\, k}f\big| \leq \sum_{k\geq 0}  r_n^{-(k+1)}\, Q^{\, k}f.$$
Since $\cB$ is a Banach lattice, the last inequality is true in norm, that is 
$$\big\|(\lambda_nI - Q)^{-1}f\big\| \leq \big\|\sum_{k\geq 0}  r_n^{-(k+1)}\, Q^{\, k}f\big\|$$
from which we deduce that $\lim_n \|(r_nI - Q)^{-1}\|_\cB = +\infty$, thus $r(Q)\in\sigma(Q)$. Finally $r(Q)$ is a finite pole of $Q$ from quasi-compactness. 

Let $q$ denote the order of the pole $r(Q)$, namely $r(Q)$ is a pole of order $q$ of the resolvent function $z\mapsto (zI-Q)^{-1}$. Then there exists $\rho>0$ such that $(zI-Q)^{-1}$ admits the following Laurent series provided that $|z-r(Q)| < \rho$ and $z\neq r(Q)$:  
$$(zI-Q)^{-1} = \sum_{k=-q}^{+\infty} (z-r(Q))^{k} A_k,$$
where $A_k$ are bounded linear operators on $\cB$. By quasi-compactness, $A_{-1}$ is a projection onto the finite subspace $\ker(Q-r(Q) I)^q$. Moreover we know that 
\begin{equation} \label{A-q-res}
A_{-q} = (Q-r(Q) I)^{q-1}\circ A_{-1} =  A_{-1}\circ(Q-r(Q) I)^{q-1}.
\end{equation}
and that, setting $r_n :=r(Q)(1 +1/n)$,   
\begin{eqnarray}
A_{-q}  &=& \lim_{n\r+\infty} (r_n-r(Q))^{q}(r_nI-Q)^{-1} \nonumber \\
&=&  \lim_{n\r+\infty} (r_n-r(Q))^{q}\sum_{k\geq 0}  r_n^{-(k+1)}\, Q^{\, k}. \label{neuman}
\end{eqnarray}
Since $Q$ is a nonnull nonnegative operator on $\cB$, so is $A_{-q}$. Since $A_{-q}\neq0$,
we take a nonnegative $h_0\in\mathcal B$ such that $\phi:= A_{-q} h_0\ne 0$ in $\cB$. 
Moreover we have $(Q-r(Q)I)A_{-q}=0$, so
$r(Q) \phi = Q\phi$. Similarly there exists a nonnegative $\psi_0\in\cB^*$ such that 
$\psi_1 := A_{-q}^*\psi_0$ is a nonzero and nonnegative element of 
$\ker(Q^*-r(Q)I)$, where $A_{-q}^*$ is the adjoint operator of $A_{-q}$.  We have $\psi_1(\phi)=\psi_1(Q\phi)/r(Q)>0$ from our hypotheses, and we set $\psi:=\psi_1/\psi_1(\phi)$.
Now assume that $q\geq 2$. Then  $A_{-q}^{\ 2}=0$ from (\ref{A-q-res}) and $A_{-1}(\cB) = \ker(Q-r(Q) I)^q$, so that $\psi_1(\phi)=(A^*_{-q}\psi_0)(A_{-q}h_0)=\psi_0(A_{-q}^2h_0)=0$. This proves by reductio ad absurdum that $p=1$. 

Observe that, from our hypotheses, we know that $\psi(h)=\psi(Qh)/r(Q)>0$ for every nonnull nonnegative $h\in\cB$.
Now let $\lambda\in\C$ and $h\in\cB$ be such that $|\lambda|=r(Q)$ and $Q h = \lambda h$.
The positivity of $Q$ gives: $|\lambda h| = r(Q) |h| = |Q h| \leq  Q |h|$. Moreover we have
$\psi(Q |h| - r(Q) |h|) = 0$. Thus $Q |h| = r(Q) |h|$ in $\cB$.

Now let us prove the last point of Proposition~\ref{firstorder}. Recall that the above nonnull nonnegative function $\phi\in\cB$ is such that $Q\phi = r(Q)\phi$. From our hypotheses we deduce that $\phi>0$. Let $\lambda\in\C$ and $h\in\cB$ be such that $|\lambda|=r(Q)$, $h\neq 0$   and $P_\gamma h = \lambda h$. 
Due to the previous point and to our assumptions, we obtain that 
$Q |h| = r(Q)\, |h|$ and $|h| = \beta \phi$ for some $\beta>0$.
One may assume that $\beta=1$ for the sake of simplicity. Then there exists a $\pi$-full and $P$-absorbing $A\in\cX$ (i.e. $\pi(A)=1$ and $P(x,A)=1$, $\forall x\in A$) such that 
\begin{subequations}
\begin{gather} 
\forall x\in A,\quad |h(x)| = \phi(x)>0 \label{E1} \\
\forall x\in A,\quad \lambda\, h(x) = \int h(y)\, e^{-\gamma\xi(y)}\, P(x,dy) \label{E2} \\
\forall x\in A,\quad r(Q)\, \phi(x) = \int \phi(y)\, e^{-\gamma\xi(y)}\, P(x,dy). \label{E3}
\end{gather}
\end{subequations}
Let $x\in A$ and define the probability measure: $\eta_x(dy) := (r(Q)\, \phi(x))^{-1}\phi(y)\, e^{-\gamma\xi(y)}\, P(x,dy)$. We have 
$$\int_\R \frac{r(Q)\, \phi(x)\, h(y)}{\lambda\, \phi(y)\, h(x)}\ \eta_x(dy) = 1.
$$
Then a standard convexity argument ensures that the following equality holds for $P(x,\cdot)-$almost every $y\in \X$: 
\begin{equation}\label{ppxy}
r(Q)\,  \phi(x)\, h(y) = \lambda\, \phi(y)\,  h(x).
\end{equation}
We have $r(Q)P\frac{h}{|h|}=\lambda \frac{h} {|h|}$ 
and so $\lambda=r(Q)$.
\end{proof}

From now on, to simplify notations, we write $R_z(\gamma):= (zI-P_\gamma)^{-1}$ for the resolvant when it is well defined. We first prove Theorems~\ref{generalspectraltheorem1} and \ref{generalspectraltheorem2} under Hypothesis \ref{hypKL}. 
Recall that $J_0:=\{t\in J : r(\gamma)>\delta_0\}$.
\begin{proof}[Proof of Theorem \ref{generalspectraltheorem1} under Hypothesis \ref{hypKL}]
The continuity on $J_0$ of $\gamma\mapsto r(\gamma) := r((P_\gamma)_{|\cB_0})$ follows from Theorem \ref{thmkellerliverani}
since $(P_\gamma)_{\gamma}$ satisfies Hypothesis~\ref{hypKL} 
with $(J,\cB_0,\cB_1)$. Moreover, due to Proposition \ref{firstorder} and to Hypothesis~\ref{hypcompl}, we know that, for every $\gamma\in J_0$, $r(\gamma)$ is the only dominating eigenvalue of $(P_\gamma)_{|\cB_0}$ and that it is a simple eigenvalue with multiplicity 1. 

Let $\chi: J_0 \rightarrow (0,+\infty)$ be defined by
$\chi(\gamma) := \max\big(\delta_0,\lambda(\gamma))$,
where we have set $\lambda(\gamma) := 
\max\{|\lambda| : \lambda\in\sigma(P_{\gamma|\cB_0})\setminus\{r(\gamma)\}\}$.
Due to Theorem \ref{thmkellerliverani}, $\chi$ is continuous on $J_0$.
Let $K$ be a compact subset of $J_0$. We set $\theta := \max_{K} \frac{\chi}{r}$. Since $\chi(\gamma) < r(\gamma)$ for every $\gamma\in K$ and since $r(\cdot)$ and $\chi(\cdot)$ are continuous, we conclude that $\theta\in(0,1)$. Next we consider any $\eta>0$ such that $\theta + 2\eta <1$. 

Let us construct the map 
$\gamma\mapsto \Pi_\gamma$, from $K$ to $\cL(\cB_0)$, and prove its properties.
Let $\gamma_0\in K$. Since $r$ is continuous on $K$, 
there exists $\varepsilon>0$ such that, for every $\gamma\in K$ such that $|\gamma-\gamma_0|\le\varepsilon$,
we have $|r(\gamma)-r(\gamma_0)|<\eta r(\gamma_0)$.
Let us write $K(\gamma_0)$ for the set of  $\gamma\in K$
such that $|\gamma-\gamma_0|\le\varepsilon$.
Observe that, for any $\gamma\in K(\gamma_0)$, 
$$\chi(\gamma)\le\theta r(\gamma) < \theta(1+\eta)r(\gamma_0)
 < (\theta +\eta)r(\gamma_0) <(1 -\eta)r(\gamma_0)  $$
and so the eigenprojector $\Pi_\gamma$ on $\ker (P_\gamma-
r(\gamma)I)$ can be defined by 
\begin{equation}\label{formulaCauchy1}
\Pi_\gamma = \frac{1}{2i\pi}\oint_{\Gamma_1(\gamma_0)} R_z(\gamma)\, dz,
\end{equation}
where $\Gamma_1(\gamma_0)$ is the oriented circle centered on $r(\gamma_0)$
with radius $\eta\, r(\gamma_0)$. Due to Theorem 
\ref{thmkellerliverani}, $\gamma\mapsto\Pi_\gamma$ is well
defined from $K(\gamma_0)$ to $\cL(\cB_0)$ and is continuous
from $K(\gamma_0)$ to $\cL(\cB_0,\cB_1)$. 

Now, for every $\gamma\in K$, we define the oriented
circle $ \Gamma_0(\gamma)  := \big\{z\in\C : |z| = (\theta + \eta)\, r(\gamma)\big\}$.
By definition of $\theta$, for every $\gamma\in K$, we have $\chi(\gamma)\leq \theta\,  r(\gamma)$ and so 
$\chi(\gamma) < (\theta + \eta)\,  r(\gamma)< r(\gamma)$.
Hence, by definition of $\chi(\gamma)$,  $R_z(\gamma)$ is well-defined in $\cL(\cB_0)$ for every $\gamma\in K$ and $z\in\Gamma_0(\gamma)$. From spectral theory, it comes that
\begin{equation}\label{formulaCauchy}
N_\gamma^n:=P_\gamma^n - r(\gamma)^n\Pi_\gamma = \frac{1}{2i\pi}\oint_{\Gamma_0(\gamma)} z^n\, R_z(\gamma)\, dz
\end{equation}
and so
\begin{equation}\label{ineg-norme}
\|P_\gamma^n - r(\gamma)^n\Pi_\gamma\|_{\cB_0} \leq B_\gamma\, \big((\theta + \eta)\, r(\gamma)\big)^{n+1}\quad \text{with}\quad B_\gamma := \sup_{|z| = (\theta + \eta)\, r(\gamma)} \|R_z(\gamma)\|_{\cB_0}.
\end{equation}
It remains to prove that
\begin{equation} \label{sup-M}
M_K := \sup_{\gamma\in K} B_\gamma < \infty.
\end{equation}
Let $\gamma_0\in K$. Since $\gamma\mapsto r(\gamma)$ is continuous at $\gamma_0$, there exists $\alpha\equiv\alpha(\gamma_0)>0$ such that, for every $ \gamma\in K$ such that $|\gamma-\gamma_0| < \alpha$, we have
$$\frac{\theta + \frac{\eta}{2}}{\theta + \eta}\, r(\gamma_0) < r(\gamma) < 
\frac{\theta + \frac{3\eta}{2}}{\theta + \eta}\, r(\gamma_0).$$
Set  $\delta := \frac{\eta}{2} \, r(\gamma_0)$. 
If $|\gamma-\gamma_0| < \alpha$ and if $|z|=(\theta + \eta)\, r(\gamma)$, we obtain  since $\delta_0\leq\chi(\gamma_0) \leq \theta\,  r(\gamma_0)$ and $\theta + 2\eta <1$: 
$$\delta_0+\delta \leq \chi(\gamma_0) +  \delta \leq \big(\theta + \frac{\eta}{2}\big) \, r(\gamma_0) < |z| < \big(\theta + \frac{3\eta}{2}\big) \, r(\gamma_0) < 
r(\gamma_0)- \delta.$$
From the previous inequalities, let us just keep in mind that 
$\chi(\gamma_0) +  \delta < |z| <  r(\gamma_0)- \delta$. Then, by definition of $\chi(\gamma_0)$, we conclude that every complex number $z$
such that $|z| = (\theta + \eta)\, r(\gamma)$ satisfies
$$|z| > \delta_0+ \delta\ \ \text{and}\ \ d\big(z,\sigma(Q(\gamma_0))\big) > \delta.$$
Hence, up to a change of $\alpha$, due to Theorem \ref{thmkellerliverani}, we obtain that
$$\sup_{\gamma>0\,:\,|\gamma-\gamma_0| < \alpha} B_\gamma = \sup\left\{ \|R_z(\gamma)\|_{\cB_0}\ :\ |\gamma-\gamma_0| < \alpha,\ |z| = (\theta + \eta)\, r(\gamma)\right\} < \infty.$$
By a standard compacity argument, we have proved
(\ref{sup-M}). Consequently, with $\theta_K := \theta + \eta$, we deduce from \eqref{ineg-norme} that 
$$\|P_\gamma^n - r(\gamma)^n\Pi_\gamma\|_{\cB_0} \leq M_K\, \big(\theta_K\, r(\gamma)\big)^n$$
from which we derive \eqref{sup-vit}. 
\end{proof}
\begin{proof}[Proof of Theorem \ref{generalspectraltheorem2} under Hypothesis \ref{hypKL}]
First we prove the following lemma. 
\begin{alem} \label{rad-egaux}
For all $\gamma\in J_0$ and for $i=1,2$, the spectral radius of $P_{\gamma| \cB_i}$ is equal to $r(\gamma):=r\big(P_{\gamma| \cB_0}\big)$.
\end{alem}
\begin{proof}
For $i=0,1,2$ set $r_i(\gamma):=r((P_\gamma)_{|\cB_i})$.
Due to Theorem \ref{generalspectraltheorem1} applied to $(P_\gamma,J,\cB_i,\cB_{i+1})$, there exists $c_i>0$
such that
$  \pi(P_\gamma^n\mathbf{1}_\X)\sim  c_i \, r_i(\gamma)^n $ as $n$ goes to infinity.
This proves the equality of the spectral radius. 
\end{proof}
We define $\chi_i$ as
$\chi$ in the proof of Theorem \ref{generalspectraltheorem1} for each $\cB_i$ ($i=0,1,2$). We define now $\chi:=\max(\chi_0,\chi_1,\chi_2)$. 

Now let us prove the differentiability of $r$ and $\Pi$.
Let  $\gamma_0 \in J_0$. Let $\eta>0$ be such that
$r(\gamma_0)>\chi(\gamma_0)+2\eta$ and let $\varepsilon>0$ 
be such that for every $\gamma\in J_0$ satisfying $|\gamma-\gamma_0|<\varepsilon$, we have 
$r(\gamma)>r(\gamma_0)-\eta>\chi(\gamma_0)+\eta>\chi(\gamma)$.
We set 
$I_0 := J_0\cap (\gamma_0-\varepsilon,\gamma_0-\varepsilon)$ and
\begin{equation} \label{D0}
\cD_0:=\{z\in\C\, :\, \chi(\gamma_0)+\eta<|z|<r(\gamma_0)-\eta\} \cup \{z\in\C : |z-r(\gamma_0)| = \eta\}.
\end{equation}
Due to the hypotheses of Theorem \ref{generalspectraltheorem2} and to an easy adaptation of \cite[Lemma A.2]{HerPen10} (see Remark~\ref{KL-explication}), we obtain that, for every $z\in\cD_0$, the map $\gamma\mapsto R_z(\gamma)$ is $C^1$ from $I_0$
to $\mathcal L(\cB_0,\cB_3)$ with $R'_z(\gamma)=R_{z}(\gamma)P'_\gamma R_{z}(\gamma)$
and
\begin{equation} \label{deri-unif}
\lim_{h\rightarrow 0}
    \sup_{z\in\cD_0} \frac{\|R_z(\gamma_0+h)-R_z(\gamma_0)-hR'_z(\gamma_0)\|_{{\cB}_0,{\cB}_3}}{|h|}=0.
\end{equation}
Moreover, for every $\gamma\in I_0$, we deduce from spectral theory that 
$$ \Pi_\gamma=\frac 1{2i\pi }\oint_{\Gamma_1} R_z(\gamma)\, dz\quad\mbox{and}\quad
     N_\gamma=\frac 1{2i\pi}
  \oint_{\Gamma_0}z R_z(\gamma)\, dz,$$
where $\Gamma_1$ is the oriented circle centered at $r(\gamma_0)$ with radius $\eta$ and $\Gamma_0$ is the oriented circle centered at $0$ with some radius 
$\vartheta_0$ satisfying $\chi(\gamma_0)+\eta<\vartheta_0<r(\gamma_0)-\eta$. Since $1_\X\in\cB_0$ by hypothesis this implies the continuous differentiability of $\gamma\mapsto N_\gamma\mathbf 1_\X$ and of $\gamma\mapsto \Pi_\gamma\mathbf 1_\X$ from $J_0$ to $\cB_3$. Since 
$r(\gamma)=\frac{(P_\gamma-N_\gamma)(\mathbf 1_\X)}{\Pi_\gamma(\mathbf 1_\X)}$ and $\gamma\mapsto P_\gamma\mathbf 1_\X$ is $C^1$ from $I_0$ to $\cB_3$ by hypothesis, we obtain the continuous differentiability of $r$ on  $I_0$. Let us define $\phi_\gamma=\Pi_\gamma\mathbf 1_\X$ and $\pi_\gamma=\Pi^*_\gamma\pi$. 
To prove that the derivative of $r$ is negative we now apply Proposition \ref{deriveenegative} with $\cB_1\hookrightarrow\cB_2$. 
Indeed $\pi_\gamma\in \cB^*_2$ since $\pi\in\cB_2^*$ and $\Pi^*_\gamma$ is well defined in $\cL(\cB_2^*)$. Moreover $\phi_\gamma\in\cB_1$ since $1_\X\in\cB_1$ and $\Pi_\gamma$ is well defined in $\cL(\cB_1)$, and $\gamma\mapsto \phi_\gamma$ is continuous from $J$ to $\cB_1$ by Theorem~\ref{generalspectraltheorem1}. Finally $\gamma\mapsto \phi_\gamma$ is differentiable from $J$ to $\cB_2$ (see the end of Remark~\ref{KL-explication}). 
\end{proof}
\begin{arem}[Proof of the differentiability of $\gamma\mapsto\Pi_ \gamma$]\label{KL-explication}
We adapt the arguments of \cite[Lemma A.2]{HerPen10}, writing 
$$R_z(\gamma) = R_z(\gamma_0) \ +\  R_z(\gamma_0)\, [P_\gamma-P_{\gamma_0}]\, R_z(\gamma_0) \ + \  \vartheta_z(\gamma),$$
$$\text{with}\qquad \vartheta_z(\gamma) := R_z(\gamma_0)\, [P_\gamma - P_{\gamma_0}]\,R_z(\gamma_0)\, [P_{\gamma} - P_{\gamma_0}]\, R_z(\gamma).$$ 
Then 
\begin{equation} \label{ineg-deri-res}
\frac{\Vert \vartheta_z(\gamma)\Vert_{{\cB}_0,{\cB}_3}}{|\gamma-\gamma_0|}
\le 
\Vert R_z(\gamma_0) \Vert_{{\cB}_2} 
\left\Vert\frac{P_\gamma - P_{\gamma_0}}{\gamma-\gamma_0}
\right\Vert_{{\cB}_1,{\cB}_2}
\Vert R_z(\gamma_0)\Vert_{{\cB}_1}
\Vert P_\gamma - P_{\gamma_0}\Vert_{{\cB}_0,{\cB}_1}
\Vert R_z(\gamma)\Vert_{{\cB}_0}.
\end{equation} 
From the hypotheses of Theorem \ref{generalspectraltheorem2} and from the resolvent bounds derived from Theorem~\ref{thmkellerliverani}, the last term goes to 0, uniformly in $z\in\cD$, when $\gamma$ goes to $\gamma_0$.  Similarly we have: 
\begin{eqnarray*}
&\ & \left\Vert R_z(\gamma_0)(P_\gamma - P_{\gamma_0})R_z(\gamma_0) - (\gamma-\gamma_0)R_z(\gamma_0) P'_{\gamma_0} R_z(\gamma_0)
\right\Vert_{{\cB}_0,{\cB}_3}   \\
&\ & \qquad \qquad \qquad \qquad \qquad \le \ M
 \Vert P_\gamma - P_{\gamma_0} - (\gamma-\gamma_0)P'_{\gamma_0}\Vert_{{\cB}_1,{\cB}_2} = \text{o}(\gamma-\gamma_0)
\end{eqnarray*}
when again the finite positive constant $M$ is derived from the resolvent bounds of Theorem~\ref{thmkellerliverani}. This shows that $R_z'(\gamma_0)=R_z(\gamma_0)P'_{\gamma_0}R_z(\gamma_0)$ in $\cL({\cB}_0,{\cB}_3)$. To prove that $\gamma\mapsto R_z'(\gamma)$ is continuous from $J_0$ to $\cL({\cB}_0,{\cB}_3)$ in a uniform way with respect to $z\in\cD$, observe that $\gamma\mapsto R_z(\gamma)$ is $\cC^0$ from $J_0$ to $\cL(\cB_0,\cB_1)$ (use Theorem~\ref{thmkellerliverani}), that $\gamma\mapsto P'_{\gamma}$ is $\cC^0$ (uniformly in $z\in\cD$) from $J_0$ to $\cL(\cB_1,\cB_2)$ by hypothesis, and finally that $\gamma\mapsto R_z(\gamma)$ is $\cC^0$ (uniformly in $z\in\cD$) from $J_0$ to $\cL(\cB_2,\cB_3)$ (again use Theorem~\ref{thmkellerliverani}). 
Observe that (\ref{ineg-deri-res}) gives the differentiability at $\gamma_0$ of the map $\gamma\mapsto R_z(\gamma)$ considered from $J$ to $\cL(\cB_0,\cB_2)$. The additional space $\cB_3$ is only required to obtain the continuous differentiability.
\end{arem}
\begin{proof}[Proof of Theorem \ref{generalspectraltheorem1} under Hypothesis \ref{hypKL}*]Here the Keller-Liverani perturbation theorem must be applied to the dual family $(P_\gamma^*)_\gamma$. Actually the hypotheses of Theorem~\ref{generalspectraltheorem1} are:   
\begin{itemize}
\item $\cB_1^* \hookrightarrow \cB_0^*$,
\item For every $\gamma\in J$, $P_\gamma^*\in \mathcal L(\mathcal B_0^*)\cap\mathcal L(\mathcal B_1^*)$,
\item $\gamma\mapsto P_\gamma^*$ is a continuous map from $J$ in $\mathcal L( \mathcal B_1^*,\mathcal B_0^*)$,  
\item There exist $\delta_0,c_0,M_0>0$ such that, for all $\gamma\in J$,  $r_{ess}\big((P_{\gamma})^*_{|\cB_1^*}\big)\le\delta_0$
and 
$$\forall n\geq 1,\ \forall f^*\in\cB_1^*,\quad 
\|(P_\gamma^*)^n f^*\|_{\cB_1^*}\le c_0(\delta_0^n\| f^*\|_{\cB_1^*} + M^n\| f^*\|_{\cB_0^*})
.$$
\item Hypothesis \ref{hypcompl} holds on $(J,\cB_1)$. 
\end{itemize}
Under these assumptions it follows from Theorem~\ref{thmkellerliverani} applied to $(P_\gamma^*)_{\gamma\in J}$ with respect to $(\mathcal B_1^*,\mathcal B_0^*)$ that, for every $\varepsilon>0$ and every $\delta>\delta_0$, the map $t\mapsto(zI-P_\gamma^*)^{-1}$ is well defined from $J_0$ to $\mathcal L(\mathcal B_1^*)$, provided that $z\in\mathcal D(\delta,\varepsilon)$ with 
$$\mathcal D(\delta,\varepsilon):=\{z\in \mathbb C,\ 
       d(z,\sigma\big((P_{\gamma_0}^*)_{|\cB_2^*})\big)>\varepsilon,\ |z|>\delta\} = \{z\in \mathbb C,\ d(z,\sigma\big((P_{\gamma_0})_{|\cB_2})\big)>\varepsilon,\ |z|>\delta\}.$$ 
In addition, the map $t\mapsto(zI-P_\gamma^*)^{-1}$, considered from $J_0$ to $\mathcal L(\mathcal B_1^*,\mathcal B_0^*)$, is continuous at every $\gamma_0\in J_0$ in a uniform way with respect to $z\in\mathcal D(\delta,\varepsilon)$. By duality this implies that $t\mapsto(zI-P_\gamma)^{-1}$ is well defined from $J_0$ to $\mathcal L(\mathcal B_1)$. Moreover, when this map is considered from $J_0$ to 
$\mathcal L(\mathcal B_0,\mathcal B_1)$, it is continuous at $\gamma_0$ in a uniform way with respect to $z\in\mathcal D(\delta,\varepsilon)$. Finally, since Hypothesis \ref{hypcompl} is assumed on $(J,\cB_1)$, Proposition~\ref{firstorder} enables us to identify the spectral elements associated with $r(\gamma) := r\big((P_\gamma)_{|\cB_1}\big)$. Consequently one can prove as above that there exists a map $\gamma\mapsto \Pi_\gamma$ from $J_0$ to $\mathcal L(\cB_1)$, which is continuous from  $J_0$ to $\mathcal L(\cB_0,\cB_1)$, such that (\ref{sup-vit}) holds with $\cB:=\cB_1$. 
\end{proof}
\begin{proof}[Proof of Theorem \ref{generalspectraltheorem2} under Hypothesis~\ref{hypKL}*] When Theorem \ref{generalspectraltheorem2} is stated with Hypothesis~\ref{hypKL}*, then Theorem \ref{generalspectraltheorem1} applies on $(\cB_0,\cB_1)$, $(\cB_1,\cB_2)$ and $(\cB_2,\cB_3)$ (with Hypothesis~\ref{hypKL}* in each case). Thus, for every $\gamma\in J_0$,  the spectral radius $r_i(\gamma):=r((P_\gamma)_{|\cB_i})$ are equal for $i=1,2,3$ (See the proof of Lemma~\ref{rad-egaux}). Observe that, from our hypotheses, Proposition~\ref{firstorder} applies to $P_\gamma$ with respect to $(J,\cB_i)$ for $i=1,2,3$. Since $P_\gamma^*$ on $\cB_i^*$ inherits the spectral properties of $P_\gamma$ on $\cB_i$, we can prove as above that, for every $\gamma_0\in J_0$ and for every $\varepsilon>0$ and $\delta>\delta_0$, the map $t\mapsto(zI-P_\gamma^*)^{-1}$ is well defined from some subinterval $I_0$ of $J_0$ containing $\gamma_0$ into $\mathcal L(\mathcal B_3^*)$, provided that $z\in\cD_0$ where the set $\cD_0$ is defined in (\ref{D0}). In addition, by applying Remark~\ref{KL-explication} with the adjoint operators $(P_\gamma^*)_\gamma$ and the spaces $\cB_3^*\hookrightarrow\cB_2^*\hookrightarrow\cB_1^*
\hookrightarrow\cB_0^*$, we can prove that the map $t\mapsto(zI-P_\gamma^*)^{-1}$, considered from $J_0$ to $\mathcal L(\mathcal B_3^*,\mathcal B_0^*)$, is $\cC^1$ in a uniform way with respect to $z\in\cD_0$. By duality, this gives (\ref{deri-unif}). We conclude the differentiability of $\gamma\mapsto
\Pi_{\gamma}^*$ from $J_0$ to $\mathcal L(\mathcal B_3^*,\mathcal B_1^*)$ and so  the differentiability of $\gamma\mapsto
\Pi_{\gamma}$ from $J_0$ to $\mathcal L(\mathcal B_1,\mathcal B_3)$. 
To prove that the derivative of $r$ is negative
we apply Proposition \ref{deriveenegative} with the spaces $\cB_1$ and $\cB_3$. Note that $\pi_\gamma:=\Pi^*_\gamma\pi\in \cB^*_3$ since $\pi\in\cB_3^*$ and $\Pi^*_\gamma$ is well defined in $\cL(\cB_3^*)$. The function $\gamma\mapsto P_\gamma$ is differentiable from $J_0$ to $\cL(\cB_1,\cB_2)$, thus from $J_0$ to 
$\cL(\cB_1,\cB_3)$. We have $\phi_\gamma:=\Pi_\gamma\mathbf 1_\X\in\cB_1$ since $1_\X\in\cB_1$ and $\Pi_\gamma$ is well defined in $\cL(\cB_1)$. Moreover $\gamma\mapsto \phi_\gamma$ is continuous from $J$ to $\cB_1$ since $\Pi_\gamma$ is well defined in $\cL(\cB_1)$, continuous from $J_0$ to $\cL(\cB_0,\cB_1)$, and $1_\X\in\cB_0$. Finally
$\gamma\mapsto \phi_\gamma$ is differentiable from $J$ to $\cB_3$ since $\Pi_\gamma$ is well defined in $\cL(\cB_3)$ and differentiable from $J_0$ to $\cL(\cB_1,\cB_3)$ and $1_\X\in\cB_1$. 
\end{proof}
\subsection{A counter-example}\label{counterexample}
Assume that $(\X,d)$ is a metric space equipped with its Borel $\sigma$-algebra. 
Let $\mathcal \cL^\infty$ denote the set of bounded functions $f:\mathbb X\rightarrow\mathbb C$, endowed with the
supremum norm.
\begin{apro}
Assume that $P$ is a Markov kernel satisfying the following condition : there exists $S\in(0,+\infty)$ such that, for every $x\in\X$, the support of $P(x,dy)$ is contained in the ball $B(x,S)$ centered at $x$ with radius $S$. Assume that $\xi(y)\r0$ when $d(y,x_0)\r +\infty$, where $x_0$ is some fixed point in $\X$.
Then, for every  $\gamma\in[0,+\infty)$, the kernel $P_\gamma$ continuously acts on $\cL^\infty$ and its spectral radius $r(\gamma)=r((P_\gamma)_{|\cL^\infty})$ satisfies the following 
$$\forall \gamma\in[0,+\infty),\quad r(\gamma) = 1.$$
\end{apro}
\begin{proof}
We clearly have $r(\gamma) \leq 1$ since $P_\gamma\leq P$  and $P$ is Markov. 
For any $\beta>0$, we obtain with $f = \mathbf{1}_{[\xi \leq \beta]}$ 
$$\forall x\in\X,\quad (P_\gamma f)(x) = \int_{[\xi \leq \beta]} e^{-\gamma\xi(y)}\, P(x,dy) \geq e^{-\gamma\beta}\, P\big(x,[\xi \leq \beta]\big).$$
The set $[\xi \leq \beta]$ contains $\X\setminus B(x_0,R)$ for some $R>0$ since $\xi(y)\r0$ when $d(y,x_0)\r +\infty$. Thus, for $d(x,x_0)$ sufficiently large ($d(x,x_0)>R+S$), we 
have $P\big(x,[\xi \leq \beta]\big) = 1$, so that 
$$\|P_\gamma\|_{\cL^\infty} \geq \|P_\gamma f\|_{\cL^\infty} \geq e^{-\gamma\beta}.$$
This gives $\|P_\gamma\|_{\cL^\infty} = 1$ when $\beta \r 0$. Similarly we obtain with $f = \mathbf{1}_{[\xi \leq \beta]}$ 
\begin{eqnarray*}
\forall x\in\X\setminus B(x_0,R+2S),\quad (P_\gamma^2 f)(x) &=& \int e^{-\gamma(\xi(y)+\xi(z))}\, \mathbf{1}_{[\xi \leq \beta]}(z)\, P(y,dz)\, P(x,dy) \\ 
&\geq& e^{-\gamma\beta}\, \int e^{-\gamma\xi(y)}\, P\big(y,[\xi \leq \beta]\big) P(x,dy)\\ 
&\geq& e^{-\gamma\beta}\, \int_{\X\setminus B(x_0,R+S)} e^{-\gamma\xi(y)}\, P\big(y,[\xi \leq \beta]\big) P(x,dy)\\ 
&\geq& e^{-2\gamma\beta}
\end{eqnarray*}
and so 
$$\forall \beta>0,\quad \|P_\gamma^2\|_{\cL^\infty} 
\geq \|P_\gamma^2 f\|_{\cL^\infty} \geq e^{-2\gamma\beta}.$$
Again this provides $\|P_\gamma^2\|_{\cL^\infty} = 1$ since $\beta$ can be taken arbitrarily large. Similarly we can prove that $\|P_\gamma^n\|_{\cL^\infty} = 1$ for every $n\geq 1$, thus $ r(\gamma) = 1$. 
\end{proof}

\section*{Acknowledgements}
We wish to thank Sana Louhichi and Bernard Ycart for interesting
discussions. We are particularly grateful to Sana Louhichi 
for having asked us this question.


\end{document}